\documentclass[11pt]{article}
\usepackage[a4paper, margin=2.5cm]{geometry}
\usepackage{graphicx} 
\usepackage{amsfonts,amssymb,amsthm,amsmath}
\usepackage[dvipsnames]{xcolor}
\usepackage{multirow}
\usepackage{caption}
\usepackage{subcaption}  
\usepackage{float}    
\usepackage{tikz}
\usepackage{url}
\usetikzlibrary{arrows.meta, positioning, shapes, calc}
\usepackage{mwe}
\usepackage{amsthm}
\usepackage[linesnumbered,ruled,vlined]{algorithm2e}
\SetKwInput{Input}{Input}
\SetKwInput{Output}{Output}

\newcommand{\SG}[1]{\textcolor{red}{#1}}
\newcommand{\AR}[1]{\textcolor{Emerald}{#1}}

\newcommand{\sA}{\textcolor{RedViolet}{S_\text{A}}}
\newcommand{\sP}{\textcolor{MidnightBlue}{S_\text{P}}}
\newcommand{\scat}{\textcolor{PineGreen}{S_\text{cat}}}
\newcommand{\scatA}{\textcolor{Sepia}{S_\text{catA}}}
\newcommand{\scatI}{\textcolor{Mahogany}{S_\text{catI}}}
\newcommand{\scatAI}{\textcolor{DarkOrchid}{S_\text{catAI}}}

\newcommand{\xA}{\textcolor{RedViolet}{x_\text{A}}}
\newcommand{\xP}{\textcolor{MidnightBlue}{x_\text{P}}}
\newcommand{\xcat}{\textcolor{PineGreen}{x_\text{cat}}}
\newcommand{\xcatA}{\textcolor{Sepia}{x_\text{catA}}}
\newcommand{\xcatI}{\textcolor{Mahogany}{x_\text{catI}}}
\newcommand{\xcatAI}{\textcolor{DarkOrchid}{x_\text{catAI}}}

\makeatletter
\newcommand*\rel@kern[1]{\kern#1\dimexpr\macc@kerna}
\newcommand*\widebar[1]{%
  \begingroup
  \def\mathaccent##1##2{%
    \rel@kern{0.8}%
    \overline{\rel@kern{-0.8}\macc@nucleus\rel@kern{0.2}}%
    \rel@kern{-0.2}%
  }%
  \macc@depth\@ne
  \let\math@bgroup\@empty \let\math@egroup\macc@set@skewchar
  \mathsurround\z@ \frozen@everymath{\mathgroup\macc@group\relax}%
  \macc@set@skewchar\relax
  \let\mathaccentV\macc@nested@a
  \macc@nested@a\relax111{#1}%
  \endgroup
}
\makeatother

\title{Data-driven discovery of chemical reaction networks}
\author{Abraham Reyes-Velazquez\thanks{Corresponding author. Department of Mathematics, The University of Manchester, Oxford Road, Manchester, M13\,9PL, United Kingdom, \texttt{abraham.reyesvelazquez@manchester.ac.uk}} \and Stefan G\"{u}ttel \and Igor Larrosa \and Jonas Latz}
\date{February 2026}

\newtheorem{theorem}{Theorem} 
\newtheorem{corollary}[theorem]{Corollary}

\begin{document}

\maketitle


\begin{abstract}
     We propose a unified framework that allows for the full mechanistic reconstruction of chemical reaction networks (CRNs) from concentration data. The framework utilizes an integral formulation of the differential equations governing the chemical reactions, followed by an automatic procedure to recover admissible mass-action mechanisms from the equations. We provide theoretical justification for the use of integral formulations using analytical and numerical error bounds. The integral formulation is demonstrated to offer superior robustness to noise and improved accuracy in both rate-law and graph recovery when compared to other commonly used formulations. Together, our  developments advance the goal of fully automated, data-driven chemical mechanism discovery.
\end{abstract}


\bigskip

\section{Introduction}
The elucidation of chemical mechanisms from data is a central challenge in chemical reaction network (CRN) theory. Traditionally, chemists have relied on empirical kinetic analysis to infer rate laws and the ordinary differential equations (ODEs) describing the time evolution of chemical species. The linking of the inferred rate laws to the underlying stoichiometric reaction equations typically involves heuristics and  expert reasoning. Among the classical approaches to model identification, some of the most influential include the method of initial rates \cite{Espenson1995}, the delplot technique \cite{bhore1990delplot}, reaction progress kinetic analysis (RPKA) \cite{blackmond2005reaction}, and variable time normalization analysis (VTNA) \cite{bures2016variable}. Despite their utility, these methods require prior mechanistic assumptions, substantial expert interpretation, and carefully designed experiments. Moreover, they typically yield only empirical rate laws, leaving the reconstruction of the underlying reaction network as a largely heuristic process.

Recent advances in machine learning and numerical linear algebra have led to a variety of computational methods for identifying chemical dynamical systems from time-series concentration data. For example, \cite{bures2023organic} employs deep neural networks trained to classify mechanisms from kinetic data; however, this approach can only discover  CRNs seen during training. Works such as \cite{burnham2007identifying, searson2007inference} aim to automatically infer ODEs of chemical reaction networks directly from experimental data using regression-based optimization. The paper \cite{willis2016inference} proposes a semi-automatic method for CRN identification by formulating the inference problem as a mixed-integer linear programming (MILP) optimization. This approach requires a predefined set of candidate reactions and uses MILP to select the minimal subset that best explains the observed concentration time series, estimating both network structure and rate constants simultaneously. While powerful in systematically exploring combinatorial reaction spaces and enforcing stoichiometric and kinetic constraints, its effectiveness depends on the quality of candidate reactions and can be computationally expensive for large or complex networks.
Although most of the focus on both modelling and inverse problems has been directed towards the continuous ODE formulation of chemical kinetics, there have also been advancements in the modelling and recovery on the discrete setting. In particular, the paper \cite{zhang2019learning} utilises stochastic population data, coming from a Monte Carlo simulation such as Gillespie's Stochastic Simulation Algorithm (SSA), and realises the recovery in this setting.

The sparse identification of nonlinear dynamics (SINDy) framework \cite{brunton2016discovering} has emerged as a particularly versatile tool due to its ability to infer parsimonious ODE models for nonlinear systems. Several adaptations of SINDy for chemical kinetics \cite{mangan2016inferring, hoffmann2019reactive, bhatt2023sindy} have demonstrated promise for data-driven discovery of rate laws. Obtaining an ODE system using SINDy is generally less computationally expensive than MILP or other regression-based approaches. However, a major limitation of SINDy, is that it identifies governing ODEs without establishing a systematic correspondence to the underlying chemical mechanism. Consequently, automatic elucidation of CRNs from experimental or simulated data remains an open challenge.
Two major difficulties hinder the  automation of CRN discovery. First, the problem of dynamical equivalence, where distinct CRNs produce identical ODE systems under mass-action kinetics~\cite{szederkenyi2010computing, szederkenyi2011inference}, complicates the recovery of unique reaction structures. Second, conservation laws inherent to many CRNs introduce numerical challenges, such as matrix rank deficiencies in the underlying linear algebra problems. These problems can often be mitigated by fusing data from multiple experiments or initial conditions, but to date there is no analysis to justify this theoretically. 

The use of integral formulations can significantly improve the recovery of CRNs from (noisy) data. Recently, several authors have explored integral formulations or hybrid approaches combining integral constraints with differential forms within the SINDy framework to improve robustness to noise and mitigate numerical differentiation errors. For example, the work \cite{forootani2023robust} augments the classical SINDy objective function with a Runge–Kutta consistency penalty (an integral-like constraint) while learning an implicit neural representation, although this work does not provide formal error bounds or compare purely integral versus purely differential formulations. Similarly, \cite{schaeffer2017sparse} directly formulates a sparse regression problem using integral terms rather than derivatives, demonstrating improved noise resilience but without detailed error analysis. A more closely related study is \cite{wei2022sparse}, which proposes an integral SINDy (ISINDy) strategy combining penalised spline smoothing with discretised integral regression; however, it also lacks formal error analysis.

Addressing these challenges requires integrating data-driven system identification with chemical network theory in a principled manner. This will be the subject of this work. We present a unified framework for data-driven discovery of chemical reaction networks that extends the SINDy methodology beyond ODE identification to full mechanistic reconstruction. Specifically, our paper contributes the following.
\begin{itemize}
    \item Automated linkage between inferred ODEs and CRNs: We introduce an algorithmic post-processing step that maps a sparse ODE model obtained by SINDy to an admissible chemical mechanism consistent with mass-action kinetics via a convex optimisation problem. For closed (reversible or weakly reversible) networks, this procedure is fully automated; for open systems, minimal user input is required to select among alternative filtration schemes prior to the same automated reconstruction.
    
    \item Differential versus integral SINDy formulations: We formulate and compare differential and integral variants of the SINDy regression problem, where concentration data are respectively differentiated or integrated after spline interpolation. Through numerical experiments, we demonstrate that the integral formulation offers superior robustness to noise and yields more accurate recovery of both rate laws and reaction structures.
    
    \item Error analysis of both formulations: We derive error bounds for each formulation and show, both analytically and empirically, that the integral SINDy approach accumulates lower numerical and regression error, supporting its improved performance for chemical kinetics data. 
\end{itemize}
Together, these contributions advance the goal of fully automated, data-driven chemical mechanism discovery by unifying sparse system identification, chemical network reconstruction, and rigorous error analysis within a single framework.

The remainder of this paper is structured as follows. Section~\ref{sec:crn} briefly reviews the formalism of chemical reaction networks (CRNs), mass-action kinetics, and challenges in reconstructing reaction graphs from observed dynamics. Section~\ref{sec:sindy} develops the SINDy-based identification methodology for CRNs, derives differential and integral variants using piecewise cubic-spline interpolation, and presents corresponding error analyses. Section~\ref{sec:graph} describes our algorithmic scheme for graph recovery, mapping the SINDy-inferred ODE model to a candidate reaction network graph consistent with mass-action kinetics. Section~\ref{sec:numex} presents numerical experiments on two representative chemical systems under noise-free and noisy data to assess ODE recovery and graph inference performance for both formulations. Finally, Section~\ref{sec:concl} concludes with a summary of contributions, limitations, and directions for future work.

\section{Chemical reaction networks}\label{sec:crn}
A chemical reaction network (CRN) is a many-body dynamical system describing the interactions and time-evolution of multiple chemical species in a well-mixed environment. CRNs are widely used in chemistry, biology, and epidemiology to represent population dynamics driven by discrete events. Formally, a CRN is defined by a triple \( (\mathcal{S},\mathcal{Q},\mathcal{R})\), where \(\mathcal{S}=\{S_{\alpha}\}_{\alpha=1}^{M}\) is the set of species, \(\mathcal{Q} = \{q_{i}\}_{i=1}^{N}\) is the set of complexes containing \emph{all possible} linear combinations \(q_{i} = \sum_{\alpha=1}^{M}Q_{\alpha,i} S_{\alpha}\) with integer weights \(Q_{\alpha,i}\geq 0\) such that \(\sum_{\alpha=1}^{M} Q_{\alpha,i} \leq p\) for some total degree \(p\) and with at least one $Q_{\alpha,i}$ nonzero for each $i$. The degree $p$ corresponds to the highest reaction order in the network and we have the relation
\[
    N = \binom{M+p}{p}-1.
\]
Finally, \(\mathcal{R}\) is the set of directed reactions \((q_i,q_j)\) where each ordered pair represents the transformation of complex~\(q_{i}\) into complex~\(q_{j}\).

When the population of a CRN is sufficiently large, the system can be described in terms of continuous time-dependent concentrations \(x_{\alpha}(t)\), leading to deterministic mass-action kinetics.  Let \(\mathbf{x}(t)=[x_{1}(t),x_{2}(t),\ldots,x_{M}(t)]^{T} \in \mathbb{R}^{M}\) be the species concentrations vector. Under mass-action kinetics, the rate \(r_{i,j} \) of a reaction \((q_i,q_j)\in\mathcal{R}\) is proportional to the product of the reactant concentrations,
\[
    r_{i,j}(\mathbf{x}(t)) = k_{i,j} \prod_{\alpha=1}^{M} x_{\alpha}^{Q_{\alpha,i}}(t) \,,
\]
with \(k_{i,j} \geq 0\) the reaction rate constant, and with \(Q_{\alpha,i}\) the stoichiometric coefficient of species \(S_{\alpha}\) in complex \(q_{i}\). The contribution of reaction \((q_{i},q_{j})\) to the time evolution of species \(\beta\) is given by 
\[
    \kappa_{\beta}^{(i,j)}(\mathbf{x}(t)) = \bigl[Q_{\beta,j} - Q_{\beta,i}\bigr]r_{i,j}(\mathbf{x}(t))  = k_{i,j}\bigl[Q_{\beta,j} - Q_{\beta,i}\bigr]\prod_{\alpha = 1}^{M} x_{\alpha}^{Q_{\alpha,i}}(t) .
\]
Note that the kinetic term \(\kappa_{\beta}^{(i,j)}\) will always be non-negative for every reaction in which species $\beta$ is not included in the reactant complex \(q_i\).

A CRN is said to be \textit{reversible} if \( (q_i, q_j) \in \mathcal{R} \) implies \( (q_j, q_i) \in \mathcal{R} \); it is said to be \textit{weakly reversible} if for every reaction \( (q_i, q_j) \in \mathcal{R} \) there exists a path \(\{(q_j, q_1), (q_1, q_2), \ldots,(q_{\ell-1}, q_\ell), (q_\ell, q_i) \}\) from \( q_j \) to~\( q_i \). In the following, we refer to both reversible and weakly reversible networks as \textit{closed}. Conversely, a CRN that is neither reversible nor weakly reversible is called \emph{open}. Examples of these types of networks are illustrated in Figure~\ref{fig:crn-types}.

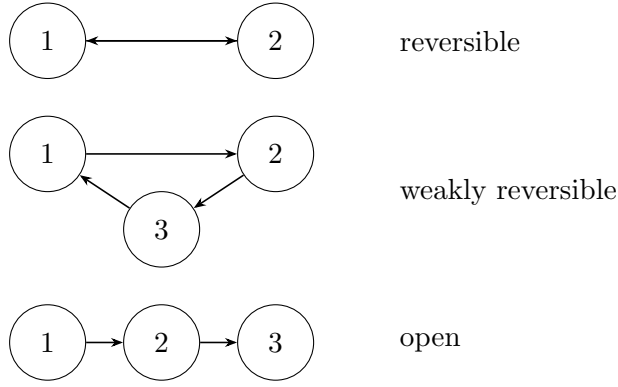
\begin{figure}[h]
    \centering
    \begin{tikzpicture}[
        node/.style={circle, draw, minimum size=1cm},
        arrow/.style={-{Stealth[scale=0.8]}, semithick}
        ]
    \node[node] (A1) at (0,0) {1};
    \node[node] (B1) at (3.0,0) {2};
    \draw[arrow] (A1) -- (B1);
    \draw[arrow] (B1) -- (A1);
    \node[anchor=west] at (4.5,0) {reversible};

    \node[node] (A2) at (0,-1.5) {1};
    \node[node] (B2) at (3.0,-1.5) {2};
    \node[node] (C2) at (1.5,-2.5) {3};
    \draw[arrow] (A2) -- (B2);
    \draw[arrow] (B2) -- (C2);
    \draw[arrow] (C2) -- (A2);
    \node[anchor=west] at (4.5,-2.0) {weakly reversible};

    \node[node] (A3) at (0,-4) {1};
    \node[node] (B3) at (1.5,-4) {2};
    \node[node] (C3) at (3,-4) {3};
    \draw[arrow] (A3) -- (B3);
    \draw[arrow] (B3) -- (C3);
    \node[anchor=west] at (4.5,-4) {open};
    \end{tikzpicture}
\caption{Examples of CRNs being reversible, weakly reversible, and open}
\label{fig:crn-types}
\end{figure}

Let \(\mathbf{Q} = [Q_{\alpha,i}] \in \mathbb{R}^{M \times N}\) be the complexes stoichiometry matrix and denote by 
\[
    \mathbf{d}(\mathbf{x}(t))=\bigg[\prod_{\alpha = 1}^{M} x_{\alpha}^{Q_{\alpha,1}}(t),\prod_{\alpha = 1}^{M} x_{\alpha}^{Q_{\alpha,2}}(t),\ldots,\prod_{\alpha = 1}^{M} x_{\alpha}^{Q_{\alpha,N}}(t) \bigg]^{T}  \in \mathbb{R}^{N}
\] 
the mass-action dictionary vector whose entries are the monomial contributions of each complex \(q_{i}\). Each entry of \(\mathbf{d}(\mathbf x(t))\) is an element in the set of all possible monomial combinations of the species concentrations \(\mathbf{x}(t)\) up to total degree \(p\).  Further, let
\[
\mathbf{K} = 
    \begin{bmatrix}
        -\sum_{i=1}^{N} k_{1,i} & k_{1,2} & \ldots & k_{1,N} \\ 
        k_{2,1} &  -\sum_{i=1}^{N} k_{2,i} & \ldots & k_{2,N} \\
        \vdots & \vdots & \ddots & \vdots \\
        k_{N,1} & k_{N,2} & \ldots & -\sum_{i=1}^{N} k_{N,i} 
    \end{bmatrix}^{T} \in \mathbb{R}^{N\times N} 
\]
be a Kirchhoff matrix which encodes the graph structure of the CRN. This matrix is such that all its columns sum to zero, all its off-diagonal entries are non-negative, and all its diagonal entries are non-positive.

Finally, the dynamical behaviour of a CRN is characterised by the (generally nonlinear) ODE system
\begin{equation} \label{eq:crn-ode}
    \dot{\mathbf{x}}(t) = \mathbf{Q}\mathbf{K}\mathbf{d}(\mathbf{x}(t)) = \mathbf{f}(\mathbf{x}(t)), \quad \mathbf{x}(0) \text{ given,}
\end{equation}
or its equivalent Picard integral formulation
\begin{equation} \label{eq:crn-integral}
    \mathbf{x}(t) -  \mathbf{x}(0) = \int_{0}^{t} \mathbf{Q}\mathbf{K}\mathbf{d}(\mathbf{x}(t)) \ dt = \int_{0}^{t} \mathbf{f}(\mathbf{x}(t)) \ dt, 
\end{equation}
where \(\mathbf{x}(0)\) is the initial state of the CRN, and \(\mathbf{f}(\mathbf{x}(t))\) is a column vector with \(M\) entries, each a polynomial function of the variables \(\mathbf{x}(t)\) with general form
\[
    \dot{x}_{\alpha}(t) = f_{\alpha}(\mathbf{x}(t))  = \sum_{\substack{\text{all reactions}\\(i,j)}} \kappa_{\alpha}^{(i,j)}(\mathbf{x}(t)),
\]
that is, they are the sum of all the kinetics \( \kappa_{\alpha}^{(i,j)}(\mathbf{x}(t)) \) that involve species \(\alpha\). Hence, all positive terms of \(f_{\alpha}(\mathbf{x}(t))\) correspond to reactions where species \(\alpha\) gains mass, and all negative terms of \(f_{\alpha}(\mathbf{x}(t))\) correspond to reactions where species \(\alpha\) loses mass. It can be shown that a polynomial ODE system with general form~(\ref{eq:crn-ode}) describes a chemically valid mass-action CRN if and only if, for each species \(\alpha\), every negative monomial in
\[
    \dot{x}_{\alpha}(t) = f_{\alpha}(\mathbf{x}(t))
\]
contains the variable~$x_{\alpha}(t)$ raised to a nonzero power~\cite{feinberg2019foundations}. In other words, in a chemically valid mass-action CRN, species only lose mass through reactions \((q_i,q_j)\) where their reactant stoichiometric coefficient \(Q_{\alpha,i}\) is greater than their product stoichiometric coefficient \(Q_{\alpha,j}\).

The coefficient matrix \(\mathbf{C} = \mathbf{Q}\mathbf{K} \in \mathbb{R}^{M\times N}\) relates the time-evolution of concentrations with the complexes in the network, and it is the object we aim to recover from time-series measurements of species concentrations; see also \eqref{eq:crn-ode}. Since the dictionary vector \(\mathbf{d}(\mathbf x(t))\) includes all monomial combinations of the species' concentrations up to total degree \(p\), and most of these combinations do not contribute to the reaction, the matrix \(\mathbf{C}\) is typically very sparse. Notably, the active (nonzero) columns of \( \mathbf{C} \) directly reveal all source complexes in the network; this is advantageous for data-driven recovery of closed CRNs where every complex is a source complex.

In both open and closed networks, a recurring phenomenon is the presence of conservation laws among species. These conservation laws are linear relations between the time-derivatives of concentrations,
\[
    \sum_{\alpha \in \Gamma_{\ell}} \dot{x}_{\alpha}(t) = 0 ,
\]
where the \(\Gamma_{\ell}\) are index sets of species with conserved mass, known as moieties. Integrating these conservation laws yields constants
\[
    \sum_{\alpha \in \Gamma_{\ell}} x_{\alpha}(t) = \sum_{\alpha \in \Gamma_{\ell}} x_{\alpha}(0)
\]
fixed by the initial conditions of the system. Conservation laws introduce linear dependencies between the concentrations \(\mathbf{x}(t)\), which in turn induce polynomial relations among the entries of~\(\mathbf{d}(\mathbf x (t))\). This can lead to issues in the recovery of the ODE system \eqref{eq:crn-ode} or \eqref{eq:crn-integral} using linear combinations of dictionary elements, as the same equation might have various identical representations.

\section{SINDy for chemical reaction networks} \label{sec:sindy}
In this section we introduce and analyse the SINDy framework in the context of the chemical reaction network dynamical problems
\begin{align*}
    \dot{\mathbf{x}}(t) &= \mathbf{C}\mathbf{d}(\mathbf{x}(t)) , \quad \mathbf{x}(0) \text{ given,}
    \\ 
    \mathbf{x}(t) -  \mathbf{x}(0) &= \int_{0}^{t} \mathbf{C}\mathbf{d}(\mathbf{x}(t)) \ dt .
\end{align*} 
Consider a CRN with \(M\) species whose concentrations \(x_{\alpha}(t)\) evolve through time. Assume we have measured the trajectories of every species in the network at time points \(0 = t_{0} < t_1 < \ldots < t_{n}\) and denote the corresponding values as
\[
    \overline{x}_{\alpha,i} = x_{\alpha}(t_{i}) + \xi_{\alpha,i}, 
\]
where $\xi_{\alpha,i}$ denotes a (small)  measurement noise for species \(\alpha\) at timestep \(t_{i}\). We assume the $\xi_{\alpha,i}$ to be i.i.d.\ Gaussian. We collect these values in an $M \times (n+1)$ data matrix \({\mathbf{X}} = [\overline{x}_{\alpha,i}]\). 
Using this data, we construct an \(N\times (n+1)\) dictionary matrix \({\mathbf{D}} = \mathbf{D}(\mathbf{X})\) whose rows are entry-wise polynomial combinations of the rows of \(\mathbf{X}\) (with polynomials of total degree $\leq p$). 
Then we have the relation 
\[
    \mathbf{C} \mathbf{D} = 
    \mathbf{f}(\mathbf{X}) \in \mathbb{R}^{M \times (n+1)},
\]
where $\mathbf f$ is applied column-wise to $\mathbf{X}$.

The structure of \(\mathbf{D}\) arises from the law of mass-action which offers a notable advantage when applying the SINDy framework to CRNs: it removes the need to specify an overly redundant dictionary. Furthermore, as in practice reactions with three reactant molecules are extremely rare, and most of these complex complex chemical processes can be reduced to systems of monomolecular and bimolecular reactions \cite{steinfeld1999chemical, upadhyay2006chemical}, a quadratic \(p=2\) dictionary is generally sufficient to capture the system's dynamics.

\subsection{Numerical differentiation vs integration}
Now, from a numerical perspective, we have two ways of approaching the data-driven discovery problem. As we only have access to time series data of the concentrations, we require to either numerically differentiate our data matrix \(\mathbf{X}\) to approximate the differential dynamical system~(\ref{eq:crn-ode}), or to numerically integrate our dictionary matrix \(\mathbf{D}\) to approximate the integral dynamical system~(\ref{eq:crn-integral}). To do this, we introduce time differentiation and integration integration  matrices denoted by \(\mathbf{L} \in \mathbb{R}^{(n+1) \times (n+1)}\) and \(\mathbf{J} \in \mathbb{R}^{(n+1) \times (n+1)}\), respectively. These operators are constructed using cubic spline interpolation of canonical basis vectors as follows: let \(\{t_{k}\}_{k=0}^{n}\) be a sequence of knots, and for each \(i \in \{0, \ldots, n\}\), define the not-a-knot scalar-valued interpolant cubic spline \(s_i(t)\) as
\[
    s_i(t) = \min_{s \in \mathbb{S}_3} \Biggl[ \sum_{k=0}^{n} \bigl( \delta_{i,k} - s_{i,k}  \bigr)^{2} \Biggr] \,,
\]
where \(\mathbb{S}_3\) denotes the space of all cubic splines with not-a-knot boundary conditions at \(t_0\) and \(t_n\).  
We then define the matrices \(\mathbf{L}\) and \(\mathbf{J}\) row-wise as:
\begin{align*}
    \mathbf{L}_{i,:} &= \left[ s_i'(t_0),\, s_i'(t_1),\, \ldots,\, s_i'(t_n) \right], \\
    \mathbf{J}_{i,:} &= \left[ \int_{t_0}^{t_0} s_i(t) \, dt,\ \int_{t_0}^{t_1} s_i(t) \, dt,\ \ldots,\ \int_{t_0}^{t_n} s_i(t) \, dt \right].
\end{align*}
Using these matrices, we can approximate
\begin{eqnarray}
    \mathbf{X}\mathbf{L} &=& \displaystyle\frac{d}{dt}\mathbf{X} + \mathbf{E}_{\text{dif}} \,, \label{eq:Edif}\\ 
    \mathbf{D}\mathbf{J} &=& \displaystyle\int_{t_0}^{t_n} \mathbf{D} \ dt +  \mathbf{E}_{\text{int}}\,,\label{eq:Eint}
\end{eqnarray}
where \(\mathbf{E}_{\text{dif}} \in \mathbb{R}^{M\times(n+1)}\) and \(\mathbf{E}_{\text{int}} \in \mathbb{R}^{N\times(n+1)}\) are matrices containing the errors incurred by numerical differentiation and integration, respectively.

\subsection{Dealing with rank deficiencies}

We can write the spline-approximated differential and integral dynamical problems as
\begin{align*}
    \mathbf{X}\mathbf{L} &= \mathbf{C}_{\text{dif}} \mathbf{D}  \,, \\ 
    \mathbf{X} - \mathbf{X}_{\text{IVP}} &= \mathbf{C}_{\text{int}} \mathbf{D}\mathbf{J} \,,
\end{align*}
where \(\mathbf{X}_{\text{IVP}} \in \mathbb{R}^{M \times (n+1)}\) is the initial value matrix of the system with each column a copy of \(\mathbf{X}[:,t_0]\). The unregularised model recovery, for each formulation, consists of solving the respective optimisation problem
\begin{align*}
    \mathbf{C}_{\mathrm{dif,ls}} &= \min_{\mathbf{C} \in \mathbb{R}^{M\times N}}\| \mathbf{X}\mathbf{L} - \mathbf{C} \mathbf{D} \|_{F}^{2} \,, \\ 
    \mathbf{C}_{\mathrm{int,ls}} &= \min_{\mathbf{C} \in \mathbb{R}^{M\times N}} \| \mathbf{X} - \mathbf{X}_{\text{IVP}} - \mathbf{C} \mathbf{D}\mathbf{J} \|_{F}^{2} 
\end{align*}
in the Frobenius matrix norm $\|\cdot\|_F$. This corresponds to a maximum likelihood estimate. 
If the dictionary matrix \(\mathbf{D}\) were of full row rank~$N$, we could uniquely determine minimising coefficient matrices 
\begin{align*}
    \mathbf{C}_{\mathrm{dif,ls}}  &= \mathbf{X}\mathbf{L}\mathbf{D}^{\dagger}\,, \\
    \mathbf{C}_{\mathrm{int,ls}}  &= \mathbf{X}_{0}\bigl[\mathbf{D}\mathbf{J}\bigr]^{\dagger} \,, 
\end{align*}
where \(\mathbf{A}^{\dagger} = \mathbf{A}^{T} (\mathbf{A}\mathbf{A}^{T})^{-1}\) denotes the Moore--Penrose pseudoinverse of a short and wide matrix~\( \mathbf{A}\). However, as mentioned above, the dictionary \(\mathbf{D}\) may be rank-defiecient, e.g., due to the presence of conservation laws in CRNs. The rank may be increased by adding measured concentrations for multiple experiments with different initial conditions. 
To be more precise, we  consider a number of $w$ experiments, each performed with distinct initial conditions such that the total mass of each moiety differs across realisations. The data is organised as:
\begin{align*}
    \widetilde{\mathbf{X}} &= \begin{bmatrix}
        \mathbf{X}^{(1)},  & \mathbf{X}^{(2)},  & \ldots, & \mathbf{X}^{(w)} 
    \end{bmatrix} \,, \\
    \widetilde{\mathbf{X}}_{\text{IVP}} &= \begin{bmatrix}
        \mathbf{X}^{(1)}_{\text{IVP}},  &  \mathbf{X}^{(2)}_{\text{IVP}},  & \ldots, &    \mathbf{X}^{(w)}_{\text{IVP}}
    \end{bmatrix} \,, \\
    \widetilde{\mathbf{X}}_{0} &= \widetilde{\mathbf{X}} - \widetilde{\mathbf{X}}_{\text{IVP}} \,,
\end{align*}
and then used to construct a multi-experiment dictionary
\[
    \widetilde{\mathbf{D}}= \begin{bmatrix}
        \mathbf{D}^{(1)},  & \mathbf{D}^{(2)},  & \ldots,  & \mathbf{D}^{(w)}
    \end{bmatrix}\,,
\]
to then obtain
\begin{align*}
    \mathbf{C}_{\mathrm{dif,ls}}  &= \widetilde{\mathbf{X}}\widetilde{\mathbf{L}}\widetilde{\mathbf{D}}^{\dagger}\,, \\
    \mathbf{C}_{\mathrm{int,ls}}  &= \widetilde{\mathbf{X}}_{0}\bigl[\widetilde{\mathbf{D}}\widetilde{\mathbf{J}}\bigr]^{\dagger} \,, 
\end{align*}
where the multi-experiment integration and differentiation matrices, \(\widetilde{\mathbf{J}}\) and \(\widetilde{\mathbf{L}}\), are constructed as the Kronecker products
\begin{align*}
    \widetilde{\mathbf{J}} &= \mathbf{I}_w \otimes \mathbf{J} \,, \\ 
    \widetilde{\mathbf{L}} &= \mathbf{I}_w \otimes \mathbf{L} \,,
\end{align*}
with \(\mathbf{I}_w \in \mathbb{R}^{w \times w}\) an identity matrix. Necessarily, $\text{rank}(\widetilde{\mathbf{D}}) \geq \text{rank}(\mathbf{D})$, and one could numerically check if sufficiently many experiments have been performed by verifying if $\text{rank}(\widetilde{\mathbf{D}}) = N$.

\subsection{Error analysis}\label{sec:err}

Our error analysis is performed in two parts. First, we  quantify the effect of numerical differentiation or integration of time discretized data that is perturbed by noise. Then, we study how these errors affect the least squares recovery of the coefficient matrices. Together, the results in this section provide a rigorous characterization of how sampling density, noise amplitude, and dictionary conditioning influence the accuracy of the recovered models, and allows us to properly compare the formulations differential and integral formulations introduced in the previous subsection. 
The proofs of the results in this section can be found in the Appendix.

\subsubsection{Effect of time discretization and measurement noise}

We first derive entry-wise and matrix-norm bounds for the errors introduced by cubic spline differentiation and integration when applied to noisy time-series data. For this  analysis, we assume that the noise is globally bounded. This is a usual assumption in the analysis of linear variational inverse problems; see \cite{scherzer}.

\begin{theorem}\label{thm:1}
Consider \(M\) scalar four-times continuously differentiable time series \(\{x_\alpha\}_{\alpha=1}^M \subset C^4([t_0,t_n])\).
Measurements \(x_\alpha(t_i) = x_{\alpha,i}\) are taken at equispaced times
\[
    t_i = t_0 + i h, \qquad i=0,1,\dots,n, \qquad
    h = \frac{t_n - t_0}{n}.
\]
These observations are corrupted by i.i.d.\ bounded additive noise \(\xi_{\alpha,i}\),
\[
    \bar{x}_{\alpha,i} = x_{\alpha,i} + \xi_{\alpha,i}, 
    \qquad |\xi_{\alpha,i}| \le \varepsilon < 1,
\]
independent across \(\alpha\) and \(i\).  
Let \(\mathbf{X} = [ x_{\alpha,i} ] \in \mathbb{R}^{M \times (n+1)} \) denote the data matrix and let \(\bar{\mathbf{X}} =[\bar{x}_{\alpha,i} ]\in \mathbb{R}^{M \times (n+1)} \) denote the noisy data matrix. 
From \(\bar{\mathbf{X}}\), construct a polynomial dictionary
\(\bar{\mathbf{D}}(\bar{\mathbf{X}}) \in \mathbb{R}^{N \times (n+1)}\)
containing all entry-wise monomials of total degree at most \(p\).
The rows of the corresponding noiseless dictionary \(\mathbf{D}\) are scalar time series \(\{d_\beta\}_{\beta=1}^N \subset C^4([t_0,t_n])\). 

Let \(\mathbf{L},\mathbf{J} \in \mathbb{R}^{(n+1)\times(n+1)}\) be the cubic spline
differentiation and integration matrices with not-a-knot boundary conditions and consider the  error matrices
\[
    \mathbf{E}_{\mathrm{dif}} = \dot{\mathbf{X}} - \bar{\mathbf{X}} \mathbf{L},
    \qquad
    \mathbf{E}_{\mathrm{int}} = \int_{t_0}^{t} \mathbf{D}(s)\,ds - \bar{\mathbf{D}}\mathbf{J},
\]
where $\dot{\mathbf{X}} \in \mathbb{R}^{M\times(n+1)}$ denotes the matrix of exact time derivatives evaluated at the sampling points, with entries
\[
(\dot{\mathbf{X}})_{\alpha,i} = \dot{x}_\alpha(t_i),
\]
and $\int_{t_0}^{t} \mathbf{D}(s)\,ds \in \mathbb{R}^{N\times(n+1)}$
denotes the matrix of exact cumulative integrals, with entries
\[
    \left(\int_{t_0}^{t} \mathbf{D}(s)\,ds\right)_{\beta,i}
    = \int_{t_0}^{t_i} d_\beta(s)\,ds. 
\]

Then we have the entry-wise bounds
\begin{align}
    \label{eq:errbnd_entry_dif}
    \bigl|(\mathbf{E}_{\mathrm{dif}})_{\alpha,i}\bigr|
    &\le \frac{\kappa_{\mathrm{dif}}}{n^3}
          + \varepsilon \, \|\mathbf{L}_{:,i}\|_1,\\[4pt]
    \label{eq:errbnd_entry_int}
    \bigl|(\mathbf{E}_{\mathrm{int}})_{\beta,i}\bigr|
    &\le \frac{\kappa_{\mathrm{int}}}{n^4}
          + \varepsilon C_{\beta} \, \|\mathbf{J}_{:,i}\|_1,
\end{align}
where
\begin{align*}
    \kappa_{\mathrm{dif}}
    &= \frac{9+\sqrt{3}}{216}
       \max_{t\in[t_0,t_n]} |x_\alpha^{(4)}(t)| \,(t_n-t_0)^3, \quad \alpha=1,\dots,M,
    \\[4pt]
    \kappa_{\mathrm{int}}
    &= \frac{1}{120}
       \max_{t\in[t_0,t_n]} |d_\beta^{(4)}(t)| \,(t_n-t_0)^4, \quad  \beta=1,\dots,N,
\end{align*}
and $C_\beta$ depends on the noiseless data.
\end{theorem}
\smallskip 
The bounds~\eqref{eq:errbnd_entry_dif}–\eqref{eq:errbnd_entry_int} still depend on norms involving the numerical differentiation and integration matrices $\mathbf{L}_{i,:}$ and $\mathbf{J}_{i,:}$. The following theorem shows that  $\|\mathbf{J}_{:,i}\|_1$ can be bounded independently of $n$, the number of time points, whereas  $\|\mathbf{L}_{:,i}\|_1$ cannot. Together with the slower growth of $\frac{\kappa_{\mathrm{int}}}{n^4}$ over $\frac{\kappa_{\mathrm{dif}}}{n^3}$, these results provide theoretical justification for the superiority of the integral formulation.

\begin{theorem} \label{cor:1}
Let $\{t_k\}_{k=0}^n$ be uniformly spaced knots with step size $h = t_{k+1}-t_k$ over $[t_0, t_n]$. 
Let $\mathbf{J}, \mathbf{L} \in \mathbb{R}^{(n+1)\times (n+1)}$ be the integration and differentiation matrices constructed from not-a-knot cubic splines $s_i$ as
\[
    \mathbf{L}_{i,:} = \Big[s_i'(t_0), \dots, s_i'(t_n)\Big], \qquad
    \mathbf{J}_{i,:} = \Biggl[\int_{t_0}^{t_0} s_i(t) \ dt, \dots, \int_{t_0}^{t_n} s_i(t) \ dt \Biggr].
\]

Then there exists a constant $C>0$, independent of $n$ and $h$, such that
\[
    \|\mathbf{L}_{:,i}\|_1 \le C\, n, \quad \|\mathbf{L}\|_\infty \le C\, n, \qquad
    \|\mathbf{J}_{:,i}\|_1 \le C\, (t_n-t_0), \quad \|\mathbf{J}\|_\infty \le C\, (t_n-t_0).
\]
\end{theorem}
To interpret Theorem~\ref{thm:1}, we can compare how the resulting bounds~\eqref{eq:errbnd_entry_dif}–\eqref{eq:errbnd_entry_int} for the differential and integral formulations behave with respect to the number \(n\) of spline knots and the noise intensity \(\varepsilon\). Using the bounds from Theorem~\ref{cor:1},
\begin{align*}
    |(\mathbf{E}_{\mathrm{dif}})_{\alpha,i}|
    &\leq \frac{\kappa_{\mathrm{dif}}}{n^3} + C\varepsilon n,\\[6pt]
    |(\mathbf{E}_{\mathrm{int}})_{\beta,i}|
    &\leq \frac{\kappa_{\mathrm{int}}}{n^4} + C\varepsilon C_{\beta}(t_n-t_0).
\end{align*}
And in the large \(n\) limit, \(n \rightarrow \infty \), the entrywise error bounds \eqref{eq:errbnd_entry_dif}-\eqref{eq:errbnd_entry_int} then behave as
\begin{align*}
    \bigl|(\mathbf{E}_{\mathrm{dif}})_{\alpha,i}\bigr| &= \mathcal{O}(\varepsilon\, n),\\
    \bigl|(\mathbf{E}_{\mathrm{int}})_{\beta,i}\bigr| &= \mathcal{O}(\varepsilon).
\end{align*}
So, for large $n$, the integration error is uniformly bounded, while the differentiation error grows linearly with $n$.

For non-noisy data \(\varepsilon = 0\), the large \(n\) behaviour of the entrywise bounds~\eqref{eq:errbnd_entry_dif}–\eqref{eq:errbnd_entry_int} reduces to
\begin{align*}
    \bigl|(\mathbf{E}_{\mathrm{dif}})_{\alpha,i}\bigr| &= \mathcal{O}(n^{-3}), \\
    \bigl|(\mathbf{E}_{\mathrm{int}})_{\beta,i}\bigr| &= \mathcal{O}(n^{-4}).
\end{align*}
\smallskip
It will be required for the following subsection to extend the above entrywise bounds~\eqref{eq:errbnd_F2_dif}--\eqref{eq:errbnd_F2_int} to bounds in the matrix 2-norm and Frobenius-norm.

\begin{corollary} \label{cor:2}
    Using the pointwise bounds from Theorem~\ref{thm:1}, 
    \begin{align*}
        |(\mathbf{E}_{\mathrm{dif}})_{\alpha,i}|
        &\leq \frac{\kappa_{\mathrm{dif}}}{n^3} + \varepsilon \, \|\mathbf{L}_{:,i}\|_1,\\[6pt]
        |(\mathbf{E}_{\mathrm{int}})_{\beta,i}|
        &\leq \frac{\kappa_{\mathrm{int}}}{n^4} + \varepsilon C_{\beta} \, \|\mathbf{J}_{:,i}\|_1,
    \end{align*}
    we then have the spectral norm bounds
    \begin{align}
        \label{eq:errbnd_F2_dif}
       \| \mathbf{E}_{\mathrm{dif}}\|_2 
       &\le \| \mathbf{E}_{\mathrm{dif}} \|_F 
        \le \sqrt{M(n+1)} \left( \frac{\kappa_{\mathrm{dif}}}{n^3} + \varepsilon \, \|\mathbf{L}\|_\infty \right), \\[6pt]
        \label{eq:errbnd_F2_int}
       \|\mathbf{E}_{\mathrm{int}}\|_2 
       &\le \|\mathbf{E}_{\mathrm{int}}\|_F 
        \le \sqrt{N(n+1)} \left( \frac{\kappa_{\mathrm{int}}}{n^4} + \varepsilon \, \max_\beta C_\beta \, \|\mathbf{J}\|_\infty \right).
    \end{align}
\end{corollary}

Using the estimates \(\|\mathbf{L}\|_\infty \le C\, n\) and \(\|\mathbf{J}\|_\infty \le C\, (t_n-t_0)\) from Theorem~\ref{cor:1}, then for large \(n\) the asymptotic behaviour becomes
\begin{align*}
    \|\mathbf{E}_{\mathrm{dif}}\|_F &= \mathcal{O}\biggl(\frac{1}{n^{5/2}} + \varepsilon\, n^{3/2}\biggr),\\[6pt]
    \|\mathbf{E}_{\mathrm{int}}\|_F &= \mathcal{O}\biggl(\frac{1}{n^{7/2}} + \varepsilon\,n^{1/2}\biggr).
\end{align*}
Thus, in the absence of noise (\(\varepsilon = 0\)), the Frobenius-norm error of the integral formulation decays faster than that of the differential formulation. 
In the presence of noise, the differential errors grow with \(n^{3/2}\), and the integration error remains less sensitive to the number of knots than the differentiation error, growing as \(n^{1/2}\), reflecting the smoothing effect of integration.

\subsubsection{Error bounds on the recovery}

We now study the effect of the errors characterized in the previous subsection on the recovery of the associated least-squares problems. This will lead to explicit  estimates on the resulting coefficient errors and residuals.

Let \(\mathbf{X}, \mathbf{D}\) be the data and dictionary matrix with noisy versions \(\bar{\mathbf{X}}, \bar{\mathbf{D}}\) as in Theorem~\ref{thm:1}. 
Consider the linear systems
\begin{align*}
    \dot{\mathbf{X}} &= \mathbf{C}_{\mathrm{dif}} \mathbf{D}, \quad\ \  \ \mathbf{C}_{\mathrm{dif}} \in \mathbb{R}^{M \times N}, \\
    \mathbf{X} &= \mathbf{C}_{\mathrm{int}} \mathbf{D}_{\mathrm{int}}, \quad \mathbf{D}_{\mathrm{int}}= \int_{t_0}^{t} \mathbf{D}(s) \ ds, \quad \mathbf{C}_{\mathrm{int}} \in \mathbb{R}^{M \times N}.
\end{align*}
with least-squares solutions
\begin{align*}
    \mathbf{C}_{\mathrm{dif}} &= \dot{\mathbf{X}} \mathbf{D}^{+}, \\
    \mathbf{C}_{\mathrm{int}} &= \mathbf{X}\mathbf{D}_{\mathrm{int}}^{+}
\end{align*}
where $\mathbf{D}^{+}, \mathbf{D}_{\mathrm{int}}^{+}$ denote the Moore--Penrose pseudoinverse of $\mathbf{D}, \mathbf{D}_{\mathrm{int}}$, respectively. 
In practice, we only have access to the noisy spline-approximated data $\bar{\mathbf{X}}$ and the noisy dictionary $\bar{\mathbf{D}}(\bar{\mathbf{X}})$, giving the estimates
\begin{align*}
    \bar{\mathbf{C}}_{\mathrm{dif}} &= (\bar{\mathbf{X}}\mathbf{L}) \, \bar{\mathbf{D}}^{+}, \\
    \bar{\mathbf{C}}_{\mathrm{int}} &= \bar{\mathbf{X}} (\bar{\mathbf{D}}\mathbf{J})^{+}.
\end{align*} 
Then, the coefficient errors
\begin{align*}
    \Delta \mathbf{C}_{\mathrm{dif}} = \mathbf{C}_{\mathrm{dif}} - \bar{\mathbf{C}}_{\mathrm{dif}} = \dot{\mathbf{X}}\mathbf{D}^+ - (\bar{\mathbf{X}}\mathbf{L}) \bar{\mathbf{D}}^+, \\[4pt]
    \Delta \mathbf{C}_{\mathrm{int}} = \mathbf{C}_{\mathrm{int}} - \bar{\mathbf{C}}_{\mathrm{int}} = \mathbf{X}\mathbf{D}_{\mathrm{int}}^{+} - \bar{\mathbf{X}} (\bar{\mathbf{D}}\mathbf{J})^{+},
\end{align*}
satisfy the decomposition
\begin{align*}
    \Delta \mathbf{C}_{\mathrm{dif}} = \dot{\mathbf{X}}\underbrace{ (\mathbf{D}^+ - \bar{\mathbf{D}}^+)}_{\text{data error}} + \underbrace{(\dot{\mathbf{X}} - \bar{\mathbf{X}}\mathbf{L})}_{\text{differentiation error}} \, \bar{\mathbf{D}}^+, \\[6pt]
    \Delta \mathbf{C}_{\mathrm{int}} = \underbrace{(\mathbf{X} - \bar{\mathbf{X}})}_{\text{data error}}\mathbf{D}_{\mathrm{int}}^{+} + \bar{\mathbf{X}} \underbrace{( \mathbf{D}_{\mathrm{int}}^{+} - (\bar{\mathbf{D}}\mathbf{J})^{+})}_{\text{integration error}}.
\end{align*}
Consequently, the differential and integral error matrices spectral norm are bounded as
\begin{align*}
    \|\Delta \mathbf{C}_{\mathrm{dif}}\|_2 
    &\le \|\dot{\mathbf{X}}\|_2 \, \|\mathbf{D}^+ - \bar{\mathbf{D}}^+\|_2
      + \|\dot{\mathbf{X}} - \bar{\mathbf{X}}\mathbf{L}\|_2 \, \|\bar{\mathbf{D}}^+\|_2, \\[2mm]
    \|\Delta \mathbf{C}_{\mathrm{int}}\|_2 
    &\le \|\mathbf{X} - \bar{\mathbf{X}}\|_2 \, \|\mathbf{D}_{\mathrm{int}}^+\|_2
      + \|\bar{\mathbf{X}}\|_2 \, \| \mathbf{D}_{\mathrm{int}}^+ - (\bar{\mathbf{D}}\mathbf{J})^+\|_2.
\end{align*}

Theorem~\(3.4\) in \cite{396bf6e1-ef54-3bf6-a49b-862db8404076} tells us that for any matrices \(\mathbf{A},\mathbf{B}\) such that \(\mathbf{B}=\mathbf{A}+\mathbf{E}\) with \(\mathbf{E}\) a small perturbation matrix and with \(\mathrm{rank}(\mathbf{A}) = \mathrm{rank}(\mathbf{B})\),
\[
    \|\mathbf{A}^{+} - \mathbf{B}^{+} \|_{2} \leq \frac{1+\sqrt{5}}{2} \|\mathbf{A}^{+} \|_{2} \| \mathbf{B}^{+} \|_{2} \|\mathbf{E} \|_{2} = \frac{1+\sqrt{5}}{2} \frac{\|\mathbf{E} \|_{2} }{\sigma_{\mathrm{min}}(\mathbf{A}) \sigma_{\mathrm{min}}({\mathbf{B}})}.
\]
This result immediately leads to bounds on $\Delta \mathbf{C}_{\mathrm{dif}}$ and $\Delta \mathbf{C}_{\mathrm{int}}$, respectively, which we summarize in the following theorem.

\begin{theorem}\label{thm:2}
Using the notation introduced above, we have 
\begin{align*}
    \|\Delta \mathbf{C}_{\mathrm{dif}}\|_2 
    &\le \frac{1+\sqrt{5}}{2} \|\boldsymbol{\Delta}_\xi\|_{2} \frac{\sigma_{\mathrm{max}}(\dot{\mathbf{X}})}{\sigma_{\mathrm{min}}(\mathbf{D}) \sigma_{\mathrm{min}}(\bar{\mathbf{D}})} +  \frac{\|\mathbf{E}_{\mathrm{dif}}\|_{2}}{\sigma_{\mathrm{min}}(\bar{\mathbf{D}})},  \\[4pt]
    \|\Delta \mathbf{C}_{\mathrm{int}}\|_2 
    &\le \frac{1+\sqrt{5}}{2} \frac{\|\mathbf{E}_{\mathrm{int}}\|_{2}}{\sigma_{\mathrm{min}}(\mathbf{D}_{\mathrm{int}}) \sigma_{\mathrm{min}}(\bar{\mathbf{D}}\mathbf{J}) } + \frac{\|\boldsymbol{\Xi}\|_{2}}{\sigma_{\mathrm{min}}(\mathbf{D}_{\mathrm{int}})},
\end{align*}
where \(\boldsymbol{\Delta}_\xi = \bar{\mathbf{D}} - \mathbf{D}\).
\end{theorem}

Using the bounds in Theorem~\ref{thm:2}, we can derive leading-order estimates for the coefficient errors in the differential and integral formulations. 
First, note that the noise matrices satisfy
\begin{align*}
    |(\boldsymbol{\Delta}_\xi)_{\beta,i}| &\le \varepsilon \, C_\beta, 
    &\Longrightarrow& &
    \|\boldsymbol{\Delta}_\xi\|_2 &\le \|\boldsymbol{\Delta}_\xi\|_F \le \varepsilon \, \sqrt{N n} \, C_{\mathrm{max}}, \quad C_{\mathrm{max}} = \max_{\beta} C_{\beta}\\
    |(\boldsymbol{\Xi})_{\alpha,i}| &\le \varepsilon,
    &\Longrightarrow& &
    \|\boldsymbol{\Xi}\|_2 &\le \|\boldsymbol{\Xi}\|_F \le \varepsilon \, \sqrt{M n}.
\end{align*}
Combining these with the Frobenius-norm bounds for the differential and integral errors~\eqref{eq:errbnd_F2_dif}--\eqref{eq:errbnd_F2_int}, we obtain
\begin{align}
    \|\Delta \mathbf{C}_{\mathrm{dif}}\|_2 
    &= \mathcal{O}\Bigg(
        \frac{\varepsilon \sqrt{N n} \, C_{\max} \sigma_{\mathrm{max}}(\dot{\mathbf{X}})}{\sigma_{\mathrm{min}}(\mathbf{D}) \sigma_{\mathrm{min}}(\bar{\mathbf{D}})} 
        + \frac{n^{-5/2}+ \varepsilon n^{3/2}}{\sigma_{\mathrm{min}}(\bar{\mathbf{D}})}  \Bigg),\\[8pt]
    \|\Delta \mathbf{C}_{\mathrm{int}}\|_2 
    &= \mathcal{O}\Bigg(
        \frac{n^{-7/2} + \varepsilon \sqrt{n}}{\sigma_{\mathrm{min}}(\mathbf{D}_{\mathrm{int}}) \sigma_{\mathrm{min}}(\bar{\mathbf{D}}\mathbf{J})} 
        + \frac{\varepsilon \sqrt{M n}}{\sigma_{\mathrm{min}}(\mathbf{D}_{\mathrm{int}})} 
        \Bigg).
\end{align}

Hence, for large \(n\) and in the presence of noise, and assuming that the singular values
\[
    \sigma_{\mathrm{min}}(\mathbf{D}), \quad \sigma_{\mathrm{min}}(\bar{\mathbf{D}}), \quad \sigma_{\mathrm{min}}(\mathbf{D}_{\mathrm{int}}), \quad \sigma_{\mathrm{min}}(\bar{\mathbf{D}}\mathbf{J}), \quad \sigma_{\mathrm{max}}(\dot{\mathbf{X}}),
\]
are independent of \(n\) and \(\varepsilon\),  the coefficient error grows like
\begin{equation}
    \|\Delta \mathbf{C}_{\mathrm{dif}}\|_2 = \mathcal{O}(\varepsilon n^{3/2}), 
    \quad 
    \|\Delta \mathbf{C}_{\mathrm{int}}\|_2 = \mathcal{O}(\varepsilon \sqrt{n}),
\end{equation}
showing that the recovery with the integral formulation is significantly less sensitive to noise than the differential formulation. 
In the absence of noise (\(\varepsilon = 0\)), the decay rates are \(\mathcal{O}(n^{-5/2})\) for differentiation and \(\mathcal{O}(n^{-7/2})\) for integration.

\subsection{Regularisation}
While the exact coefficient matrix $\mathbf{C}$ underlying the reaction system may be sparse, rank deficiencies in the dictionary, measurement noise, or numerical instabilities may pollute the zero entries in the recovery significantly, leading to a dense estimate. 
In the chemical reaction network setting, a dense coefficient matrix is typically not interpretable or chemically meaningful, so sparsity is preferred. Thus we seek to solve the following regularised optimisation problems
\begin{align*}
    \mathbf{C}_{\mathrm{dif,reg}}  &= \min_{\mathbf{C}} \big\| \widetilde{\mathbf{X}}\widetilde{\mathbf{L}} - \mathbf{C} \widetilde{\mathbf{D}}\big\|_{F}^{2} + \text{regularisation}, \\
    \mathbf{C}_{\mathrm{int,reg}}  &= \min_{\mathbf{C}} \big\| \widetilde{\mathbf{X}}_{0} - \mathbf{C}\widetilde{\mathbf{D}}\widetilde{\mathbf{J}}\big\|_{F}^{2} + \text{regularisation}.
\end{align*}
In practice, we solve for \(\mathbf{C}\) row-wise using a sequentially thresholded least squares (STLS) method~\cite{brunton2016discovering}, which can be regarded as a simple form of \(\ell^0\) regularisation~\cite{zhang2019convergence}. This algorithm proceeds by alternating between standard least squares estimation and coefficient thresholding to enforce sparsity. Specifically, for each row \(\mathbf{C}_{\alpha,:}\) of the coefficient matrix \(\mathbf{C} \in \mathbb{R}^{M \times N}\), the following steps are performed:

\begin{enumerate}
    \item Initialise the support set \( S = \{1, \dots, N\} \), corresponding to all candidate dictionary terms.
    
    \item Solve the least-squares problem restricted to the current support:
    \[
        \mathbf{C}_{\alpha}^{(S)} = \arg\min_{\mathbf{c} \in \mathbb{R}^{|S|}} \left\| \widetilde{\mathbf{x}}_i - \mathbf{c} \mathbf{D}_S \right\|_2^2 \,,
    \]
    where \(\widetilde{\mathbf{x}}_i\) is the \(i\)th row of \(\widetilde{\mathbf{X}}\widetilde{\mathbf{L}}\) (differential formulation) or \(\widetilde{\mathbf{X}}_0\) (integral formulation), and \(\mathbf{D}_S\) denotes the columns of the dictionary matrix  \(\widetilde{\mathbf{D}}\) (differential formulation) or \(\widetilde{\mathbf{D}}\widetilde{\mathbf{J}}\) (integral formulation) indexed by the support set \(S\).
    
    \item Threshold coefficients to promote sparsity by updating the support with the remaining indices that satisfy
    \[
        S \leftarrow \left\{ j \in S : |c_{i,j}| > \tau \right\} \,,
    \]
    where \(\tau > 0\) is a user-defined sparsity threshold.
    
    \item Repeat steps 2 and 3 until the support set \(S\) stabilizes (i.e., does not change between iterations) or a maximum number of iterations is reached.
\end{enumerate}

This procedure effectively eliminates terms with small coefficients and promotes sparse solutions. This is computationally efficient and often sufficient in practice for identifying parsimonious models~\cite{brunton2016discovering}. The result is a sparse estimate \(\mathbf{C}_{\text{stls}}\), which can then be interpreted in terms of active reaction terms or polynomial interactions.

While an error analysis along the lines of Section~\ref{sec:err} for the STLS regularised problem is outside the scope of this work, particularly due to the heuristic nature of the STLS method that we use to enforce sparsity, the unregularised analysis remains informative. The obtained error estimates provide a useful benchmark for evaluating the performance of regularised approaches and for understanding how noise and model complexity (via the number of knots) interact. The favourable properties of the integral formulation suggest that this approach may retain its advantages in regularised settings.

\section{Graph recovery from SINDy-inferred dynamics}\label{sec:graph}
We now build on the results of the previous section to demonstrate how the reaction graph of a chemical reaction network can be reconstructed from the sparse coefficient matrix \( \mathbf{C}_{\mathrm{stls}} \) obtained via sequential thresholded least squares.

Given a sparse dynamical model of the form
\[
    \dot{\mathbf{x}}(t) = \mathbf{C}_{\mathrm{stls}}\,\mathbf{d}(t),
\]
we define its \emph{effective dense representation} as
\[
    \dot{\mathbf{x}}(t) = \mathbf{C}_{\mathrm{eff}}\,\mathbf{d}_{\mathrm{eff}}(t),
\]
where \( \mathbf{C}_{\mathrm{eff}} \in \mathbb{R}^{M \times r} \) is obtained from \( \mathbf{C}_{\mathrm{stls}} \) by discarding all columns whose \(\ell^\infty\)-norm is below \( \tau \), and \( \mathbf{d}_{\mathrm{eff}}(t) \in \mathbb{R}^{r} \) contains the corresponding active components of the dictionary vector \( \mathbf{d}(t) \). Here, \( r \leq N \) denotes the number of active terms retained after this post-processing step.  

This column-based filtering acts as an additional sparsity-enforcing mechanism that isolates the most significant terms which are presumed to correspond to genuine reaction pathways in the underlying CRN. 
The support \(\mathcal{Q}_{\mathrm{source}}\) of \( \mathbf{d}_{\mathrm{eff}}(t) \) within the full complexes set \( \mathcal{Q} \) identifies the indices of the active monomials, each corresponding to a \emph{source complex} in the reaction network. Consequently, the support \(\mathcal{Q}_{\mathrm{source}}\) defines an effective source-complex stoichiometric matrix \( \mathbf{Q}_{\mathrm{eff}} \in \mathbb{R}^{M \times r} \).

Graph recovery then reduces to solving the convex optimisation problem
\begin{equation} \label{eq:kirchhoff-opt-prob}
    \min_{\mathbf{K}} \left\| \mathbf{C}_{\mathrm{eff}} - \mathbf{Q}_{\mathrm{eff}}\mathbf{K} \right\|_{F}^{2},
\end{equation}
subject to the requirement that \( \mathbf{K} \in \mathbb{R}^{r \times r} \) be a \emph{Kirchhoff matrix}, satisfying
\begin{align*}
    \sum_{k=1}^r \mathbf{K}_{k,i} &= 0 \quad &&\text{for all } i, \\
    \mathbf{K}_{i,i} &< 0 \quad &&\text{for all } i, \\
    \mathbf{K}_{i,j} &> 0 \quad &&\text{for all } i \neq j.
\end{align*}
This optimisation enforces mass-balance constraints and yields an admissible reaction-graph structure consistent with the inferred dynamics.

For closed CRMs, the above approach to graph recovery is entirely automated. In this case, the nonzero columns of \( \mathbf{C}_{\mathrm{eff}} \) directly identify all participating complexes, so that \( \mathbf{Q}_{\mathrm{eff}} \) contains a complete representation of the network. Solving problem~\eqref{eq:kirchhoff-opt-prob} therefore recovers the true reaction graph without additional user intervention.

By contrast, open CRNs pose additional challenges because \(\mathcal{Q}_{\mathrm{source}} \subset \mathcal{Q}\) typically omits some participating complexes. Nevertheless, recovery remains feasible for a broad and practically relevant class of open systems, namely those composed of a closed subgraph coupled to one or more peripheral source or sink complexes, which we term \emph{peripheral-source/sink CRNs}.  
An illustrative example is shown in Figure~\ref{fig:special-open-crns}. In such systems, the source complexes appear explicitly in \( \mathbf{C}_{\mathrm{eff}} \), so recovery proceeds identically to the closed-network case when the network is open only through sources. 

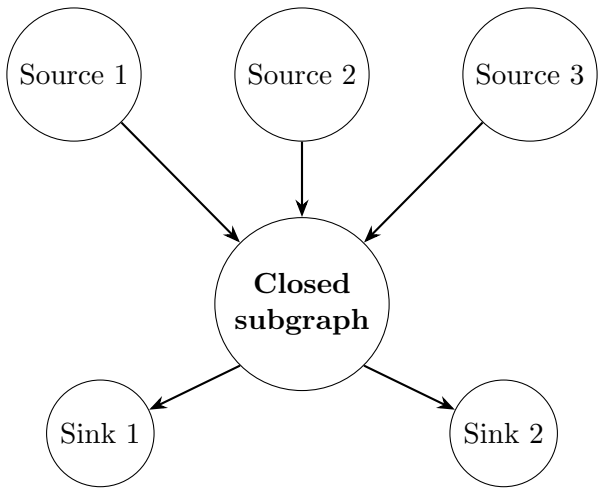
\begin{figure}[H]
    \centering
    \begin{tikzpicture}[
        node/.style={draw, circle, minimum width=0.7cm, align=center},
        source/.style={node},
        sink/.style={node},
        subgraph/.style={draw, circle, minimum width=0.7cm, font=\bfseries, align=center},
        arrow/.style={-{Stealth[scale=1]}, thick}
    ]
    \node[source] (S2) at (0,0) {Source 2};
    \node[source] (S1) at ($(S2) + (-3cm, 0)$) {Source 1};
    \node[source] (S3) at ($(S2) + (3cm, 0)$) {Source 3};

    \node[subgraph, below=1.0cm of S2] (Closed) {Closed\\subgraph};

    \node[sink, below=1.0cm of Closed.west, xshift=-1.5cm] (K1) {Sink 1};
    \node[sink, below=1.0cm of Closed.east, xshift=1.5cm] (K2) {Sink 2};

    \draw[arrow] (S1) -- (Closed.north west);
    \draw[arrow] (S2) -- (Closed.north);
    \draw[arrow] (S3) -- (Closed.north east);

    \draw[arrow] (Closed.south west) -- (K1);
    \draw[arrow] (Closed.south east) -- (K2);
    \end{tikzpicture}
    \caption{Example of an open chemical reaction network composed of a closed subgraph connected to multiple peripheral source and sink complexes. Each source feeds into, and each sink receives output from, the closed portion of the network.}
    \label{fig:special-open-crns}
\end{figure}

A subtlety arises with \emph{sink complexes}, which leave no direct signature in \( \mathbf{C}_{\mathrm{eff}} \). To address this, we introduce a synthetic \emph{zero-complex} serving as a placeholder node that represents a general sink. Operationally, this amounts to appending a zero column to \( \mathbf{Q}_{\mathrm{eff}} \) before solving problem~\eqref{eq:kirchhoff-opt-prob}. This modification enables recovery of outflows from the closed subgraph without explicit identification of each sink complex.

A practical alternative filtering strategy that often yields improved results constructs \( \mathbf{C}_{\mathrm{eff}} \) and \( \mathbf{Q}_{\mathrm{eff}} \) by retaining the \( M \times M \) block of \( \mathbf{C}_{\mathrm{stls}} \) corresponding to individual species, while applying the \(\ell^\infty\)-based thresholding only to columns representing nonlinear complexes.   
More generally, other filtering strategies, possibly incorporating expert knowledge or heuristic criteria, can be developed to tailor the construction of \( \mathbf{C}_{\mathrm{eff}} \) and \( \mathbf{Q}_{\mathrm{eff}} \) for the optimisation problem~\eqref{eq:kirchhoff-opt-prob}.

In general, this graph-recovery framework enables flexible yet minimally supervised reconstruction of open chemical reaction graphs, while preserving the fully automated mechanism discovery achievable for closed networks. In the following section, we evaluate the performance of the proposed framework through numerical experiments on representative chemical systems, designed to assess its ability to recover both kinetic laws and reaction-graph structures from concentration time-series data under noise-free and noisy conditions.

\begin{algorithm}[H]
\caption{Data-driven recovery of chemical reaction networks}
\label{alg:crn-recovery}
\Input{Multi-experiment concentrations time-series matrix \(\mathbf{X} \in \mathbb{R}^{M\times w(n+1)}\)}
\Output{Coefficient matrix \(\mathbf{C} \in \mathbb{R}^{M\times N}\), Kirchhoff matrix \(\mathbf{K} \in \mathbb{R}^{r\times r}\)}
Compute dictionary matrix \(\mathbf{D} = \mathbf{D}(\mathbf{X})\) using polynomial combinations of rows of \(\mathbf{X}\)\;
Construct \(\mathbf{X}_{\text{IVP}} \in \mathbb{R}^{M \times (n+1)}\), the initial value matrix of the system \;
Construct integration and differentiation operator matrices \(\mathbf{L}, \mathbf{J} \in \mathbb{R}^{(n+1) \times (n+1)}\) via cubic spline interpolation of canonical basis \;
Solve the optimisation problems:
\begin{align*}
    \mathbf{C}_{\mathrm{dif,stls}}  &= \arg\min_{\mathbf{C}} \big\| \widetilde{\mathbf{X}}\widetilde{\mathbf{L}} - \mathbf{C} \widetilde{\mathbf{D}}\big\|_{F}^{2} + \text{regularisation}, \\
    \mathbf{C}_{\mathrm{int,stls}}  &= \arg\min_{\mathbf{C}} \big\| \widetilde{\mathbf{X}}_{0} - \mathbf{C}\widetilde{\mathbf{D}}\widetilde{\mathbf{J}}\big\|_{F}^{2} + \text{regularisation},
\end{align*}
using sequentially thresholded least squares (STLS) with threshold parameter \(\tau\)\;
Construct effective coefficient matrix \(\mathbf{C}_{\mathrm{eff}} \in \mathbb{R}^{M \times r}\) by removing columns of \(\mathbf{C}_{\mathrm{stls}}\) whose \(\ell^\infty\)-norm is less than \(\tau\)\;
Construct effective complexes stoichiometry matrix \(\mathbf{Q}_{\mathrm{eff}} \in \mathbb{R}^{r \times r}\) using the column support of \(\mathbf{C}_{\mathrm{eff}}\) in \(\mathbf{C}_{\mathrm{stls}}\)\;
Solve the constrained optimisation problem:
\begin{equation*}
    \min_{\mathbf{K}} \left\| \mathbf{C}_{\mathrm{eff}} - \mathbf{Q}_{\mathrm{eff}}\mathbf{K} \right\|_{F}^{2}
\end{equation*}
subject to:
\begin{align*}
    \sum_{k=1}^r \mathbf{K}_{k,i} &= 0 && \text{for all } i, \\
    \mathbf{K}_{i,i} &< 0 && \text{for all } i, \\
    \mathbf{K}_{i,j} &> 0 && \text{for all } i \neq j.
\end{align*}
\end{algorithm}

\section{Numerical experiments}\label{sec:numex}
To evaluate the performance of the proposed CRN recovery framework, we apply it to four benchmark chemical reaction network systems drawn from the literature. The objectives of this section are twofold: (i) to demonstrate the accuracy of the inferred network structures, and (ii) to identify practical considerations relevant to the implementation of the procedure. The Python code and data for reproducing all experiments can be found on 
\begin{center}
\url{https://github.com/nla-group/ChemSINDy/}.
\end{center}

Two of the selected CRNs model organic catalytic processes and are taken from the catalogue in~\cite{bures2023organic}. To broaden the evaluation, we additionally include the Van de Vusse system from \cite{burnham2007identifying} and a gene transcription network from \cite{szederkenyi2011inference}, thereby covering a diverse range of kinetic regimes and network topologies. We refer to each model using the naming conventions adopted in the corresponding references.

We implemented the data-driven differentiation and integration procedures described in the previous sections in Python. The general outline of our algorithm is presented in Algorithm~\ref{alg:crn-recovery}. 
In our implementation, the data matrix \(\mathbf{X}\) is synthetically constructed. For every realisation of Algorithm~\ref{alg:crn-recovery} this is done as follows: we generate a random set of kinetic constants \(\{k_l : k_{\min} \leq k_l \leq k_{\max}\}\), drawn from a uniform distribution in the interval \([k_{\min} ,k_{\max}]\), and then we numerically integrate the corresponding ODE system
\[
    \dot{\mathbf{x}}(t) = \mathbf{F}(\mathbf{x}(t)) 
\]
over the interval \([0,t_n]\) for a set of \(w\) distinct, randomly generated, initial conditions. 

\subsection{Mechanism M1}
We begin by analysing the M1 mechanism \cite{bures2023organic}, also known as the reversible Michaelis--Menten reaction. The graph of this CRN is shown in Figure~\ref{fig:M1-Graph}.

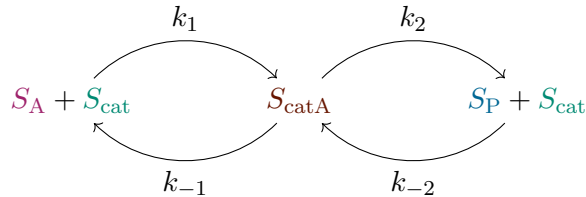
\begin{figure}[H]
    \centering
    \begin{tikzpicture}
        \node (A) at (0,0) {$\sA+\scat$};
        \node (B) at (3,0) {$\scatA$};
        \node (C) at (6,0) {$\sP+\scat$};
        
        \draw[->] (A) to[out=45,in=135]  node[above] {$k_{1}$} (B);
        \draw[<-] (A) to[out=-45,in=-135]  node[below] {$k_{-1}$} (B);
        
        \draw[->] (B) to[out=45,in=135]   node[above] {$k_{2}$} (C);
        \draw[<-] (B) to[out=-45,in=-135]   node[below] {$k_{-2}$} (C);
    \end{tikzpicture}
    \caption{Graph of the M1 mechanism}
    \label{fig:M1-Graph}
\end{figure}
The M1 mechanism comprises a set of \(4\) chemical species, \(\{\sA, \sP, \scat, \scatA\}\), whose interactions form a closed graph. As discussed previously, the graph recovery process for closed networks is straightforward, with the primary challenge being to ensure that the data collected across experiments leads to a dictionary matrix with full row rank. Numerically, we found this to be achieved by stacking \(6\) or more experiments with different initial conditions.

In all the numerical experiments of Model M1 that follow, the training data matrix \(\mathbf{X}\) used for each realisation of our procedure, Algorithm~\ref{alg:crn-recovery}, is obtained as follows. We first set the kinetic constants \(\{k_1, k_2, k_{-1}, k_{-2}\}\) by sampling each uniformly from the interval \([5 \times 10^{-2}, 1.0]\). This fixes a coefficient matrix \(\mathbf{C}_\mathrm{ex}\). Next, we generate \(w = 6\) initial condition vectors 
\[
    \mathbf{x}^{(i)}(0) = \left(\xA^{(i)}(0), \xP^{(i)}(0), \xcat^{(i)}(0), \xcatA^{(i)}(0)\right),
\]
for \(i = 1, 2, \ldots, w\), with each component drawn independently from a uniform distribution on \([0, 1]\). Finally, we numerically integrate the dynamical system defined by \(\mathbf{C}_\mathrm{ex}\) over the time interval \([0, 20]\), discretised into uniformly spaced time points, for each configuration of initial conditions \(\mathbf{x}^{(i)}(0)\).

\subsubsection{Model recovery analysis}
We begin by analysing how the errors 
\begin{align*}
    \|\Delta\mathbf{C}_{\mathrm{int,stls}}\|_{2} &= \|\mathbf{C}_{\mathrm{int,stls}} - \mathbf{C}_{\mathrm{ex}} \|_{2} \,, \ \quad \|\Delta\mathbf{C}_{\mathrm{int,ls}}\|_{2} = \|\mathbf{C}_{\mathrm{int,ls}} - \mathbf{C}_{\mathrm{ex}} \|_{2} ,
    \\ 
    \|\Delta\mathbf{C}_{\mathrm{dif,stls}}\|_{2} &= \|\mathbf{C}_{\mathrm{dif,stls}} - \mathbf{C}_{\mathrm{ex}} \|_{2} \,, \ \quad \|\Delta\mathbf{C}_{\mathrm{dif,ls}}\|_{2} = \|\mathbf{C}_{\mathrm{dif,ls}} - \mathbf{C}_{\mathrm{ex}} \|_{2} 
\end{align*}
of the integral and differential formulations behave as we vary the the number of sample points of the fixed simulation interval \([0,20]\) from 50 to 1000 in increments of 50. 
This entire procedure is repeated 100 times, each with independently sampled kinetic constants and initial conditions. The reconstruction errors from all 100 trials are shown in Figure~\ref{fig:M1_errbnd_noiseoff}, where faint lines represent individual trials and bold lines indicate the average error across repetitions. We compare the observed average errors with the theoretical decay rates established in Theorem~\ref{thm:2}. In our experiments, the errors from both formulations decay faster with respect to the number of time points than the rate predicted by theory. Nevertheless, the observation that the integration-based formulation achieves lower error and faster convergence is consistently confirmed by the numerical results.

\begin{figure}[h]
    \centering
    \includegraphics[width=\textwidth]{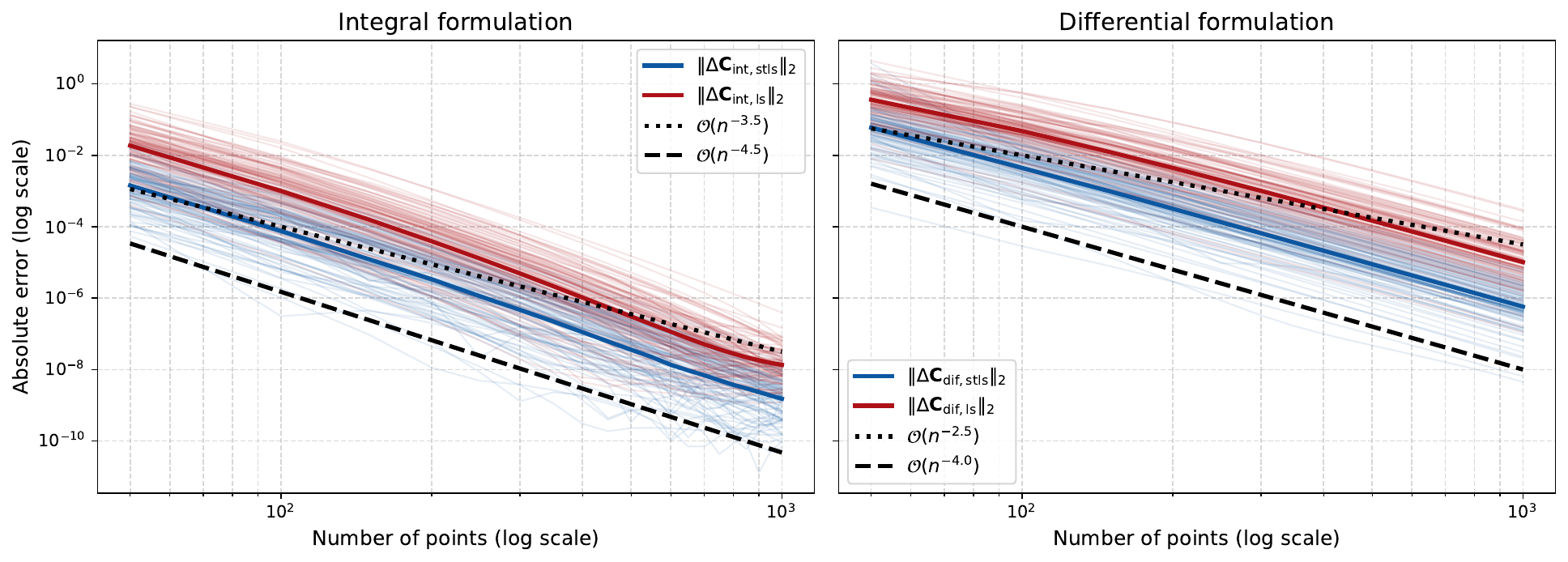}
    \caption{
    M1 reconstruction error for integration-based and differentiation-based recovery methods across increasing number of time points. Each faint line corresponds to one of 100 independent trials; bold lines represent the geometric mean of all realisations. In each subplot, we also include two dashed reference lines to assess the decay rate with respect to the number of time points. One line follows the theoretical error bound from Theorem~\ref{thm:2} and the other was fitted to the observed numerical decay rate. }
    \label{fig:M1_errbnd_noiseoff}
\end{figure}

In Figure~\ref{fig:support-mismatch} we assess the structural accuracy of the recovered coefficient matrices by quantifying their support mismatch. This metric is defined as the number of entries that differ between the recovered matrix and its ground truth, specifically, the sum of missing true nonzero entries and falsely added ones, and provides a direct measure of structural recovery performance. 

To estimate this metric statistically, we conduct 1000 independent trials in four temporal resolutions, 25, 50, 75, and 100 uniformly spaced time-points. For each trial, the synthetic data is constructed as before, at the corresponding temporal resolution, and this is used to compute matrices \(\mathbf{C}_{\mathrm{int,stls}}, \mathbf{C}_{\mathrm{dif,stls}}\). We then measure the support mismatch between each of these and the ground truth coefficient matrix \(\mathbf{C}_{\mathrm{ex}}\).

The results of these experiments are presented in Figure~\ref{fig:support-mismatch}, which compares the distribution of support mismatches across all 1000 trials for both formulations and each resolution level. In the presented plots, we restrict the histogram bins to a maximum of 10 to emphasize the more informative low-error region. We find that the integral-based algorithm identifies the correct nonzero coefficients more often than the differential-based version. This is particularly true when the number of time points is small.

We also show in Figure~\ref{fig:support-mismatch} the support mismatch in the identified Kirchhoff matrices $\mathbf{K}_{\mathrm{int,stls}}$ and $\mathbf{K}_{\mathrm{dif,stls}}$. Since the size of these matrices varies depend on which columns are selected by Algorithm~\ref{alg:crn-recovery}, we can only include those trials that produced Kirchhoff matrices of the same size as the ground truth. Nevertheless, the bar chart indicates that a good recovery of the coefficient matrices generally leads to a good recovery of the reduced Kirchhoff matrices.

\begin{figure}[H]
    \centering
    \includegraphics[width=0.90\textwidth]{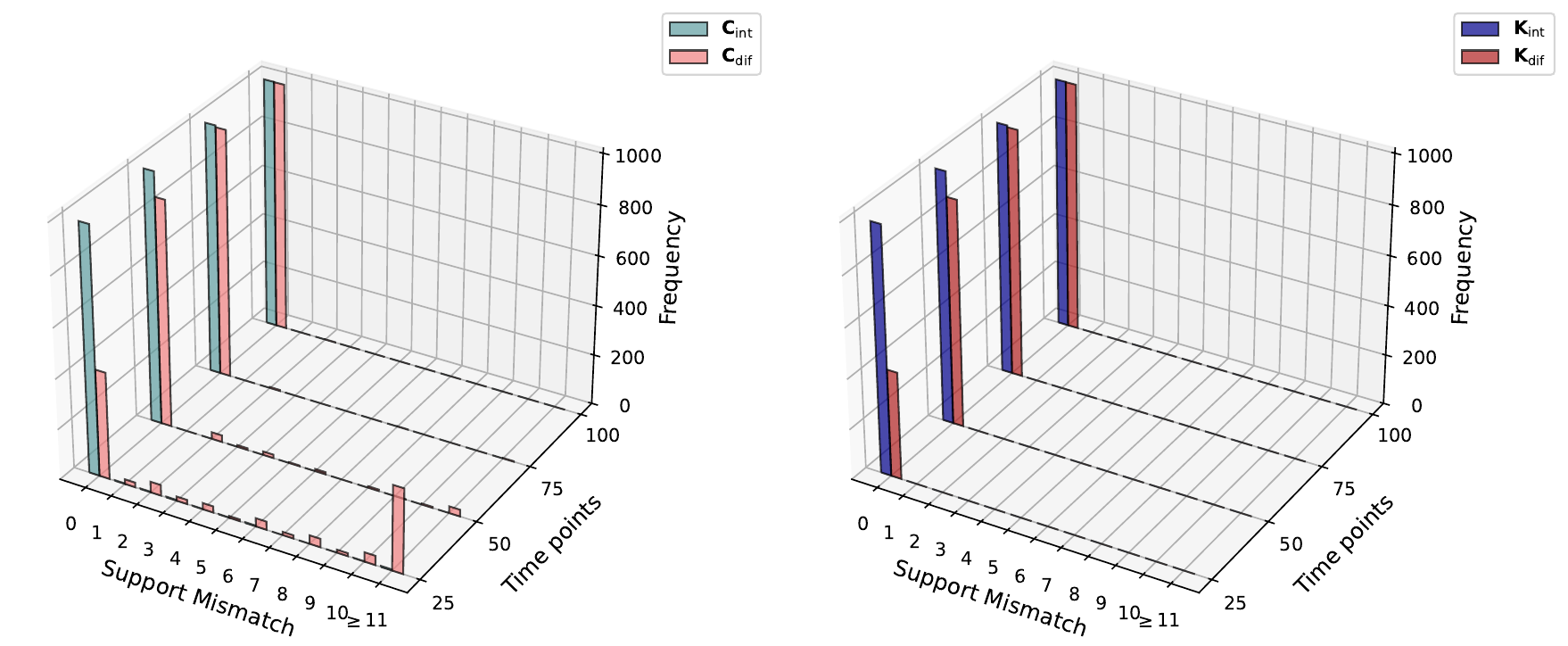}
    \caption{Left: M1 support mismatch between the ground-truth matrix \(\mathbf{C}_{\mathrm{ex}}\) and recovered coefficient matrices over 1000 trials, for both differentiation and integration-based formulations, with varying number of time points. Right: Quality of the recovered Kirchhoff matrices.}
    \label{fig:support-mismatch}
\end{figure}

\subsubsection{Graph recovery analysis}
To illustrate the overall performance of our graph-recovery formalism, we present the discovered graph for a single instance of the algorithm. We fix a set of kinetic constants \(\{k_1, k_2, k_{-1}, k_{-2}\}\), sampled uniformly from the interval \([5 \times 10^{-2}, 1.0]\), and a collection \(S_0\) of six initial conditions. Each initial condition corresponds to the concentrations of the four species \(\sA, \sP, \scat,\) and \(\scatA\), at time \(t_0 = 0\) with each value drawn independently from a uniform distribution on \([0,1]\). Using this fixed system, we generate synthetic data by integrating the associated ODE system over the time interval \([0, 20]\), discretised into \(30\) uniformly spaced time points. We then apply the integration-based recovery procedure to the resulting data, and compare the recovered network to the ground-truth network that generated the dynamics. 
This comparison is shown in Figure~\ref{fig:M1_graph-comparison}. We find that we recover the structure of our ground-truth reaction network.

\begin{figure}[H]
   \hspace*{-8mm}\includegraphics[width=1.1\textwidth]{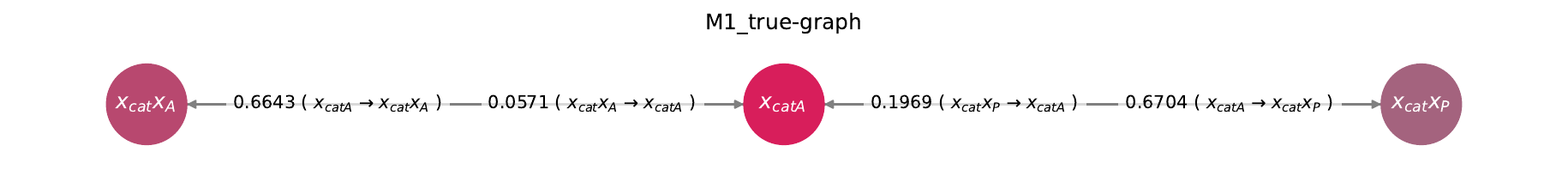}
    \hspace*{-8mm}\includegraphics[width=1.1\textwidth]{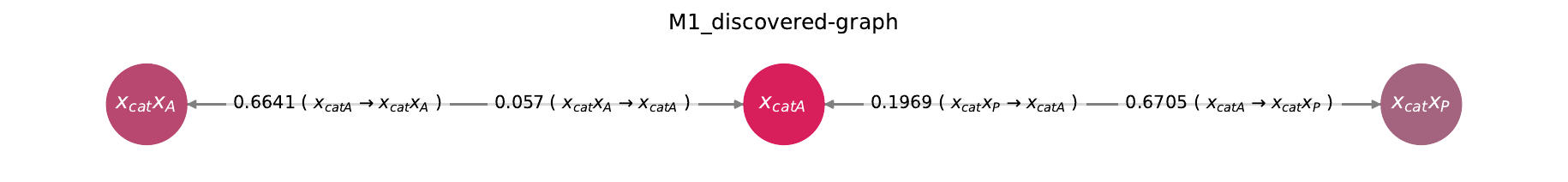}
    \caption{Top: Ground-truth CRN graph corresponding to the fixed ODE system. Bottom: Recovered CRN graph using the integration-based formulation with 30 time points.}
    \label{fig:M1_graph-comparison}
\end{figure}

\subsubsection{Noisy measurements of the M1 mechanism}
To evaluate the robustness of the recovery procedure in the presence of noise, we investigated the behaviour of the errors
\begin{align*}
    \|\Delta\mathbf{C}_{\mathrm{int,stls}}\|_{2} ,\quad
    \|\Delta\mathbf{C}_{\mathrm{dif,stls}}\|_{2} 
\end{align*}
when the data matrix \(\mathbf{X}\) is subjected to additive noise. Specifically, each entry of \(\mathbf{X}\) was independently perturbed by zero-mean Gaussian noise with variance~\(10^{-4}\). The effect of noise on the recovery errors is shown in Figure~\ref{fig:M1_errbnd_noiseon}. Again, the integral formulations leads to recoveries with smaller error, about one order of magnitude better.

\begin{figure}[h]
    \centering
    \includegraphics[width=\textwidth]{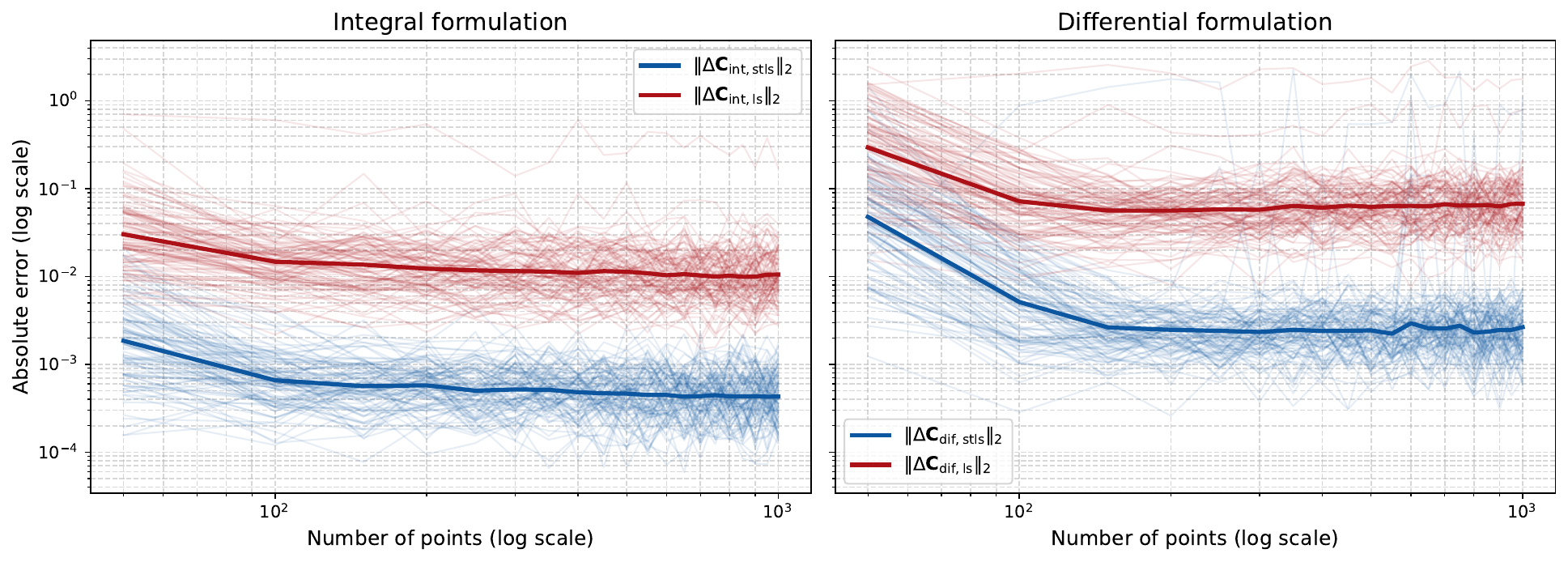}
    \caption{
    Noisy M1 reconstruction error for integration-based and differentiation-based recovery methods across increasing number of time points. Each faint line corresponds to one of 100 independent trials; bold lines represent the geometric mean of all realisations.}
    \label{fig:M1_errbnd_noiseon}
\end{figure}

Taken together, these results demonstrate that the integral formulation yields superior performance compared to the differential approach, achieving lower reconstruction errors and more accurate support recovery.

\subsection{Mechanism M20}
We now turn to the analysis of Mechanism~M20, which extends the M1 reversible Michaelis--Menten mechanism by incorporating irreversible inhibition reactions acting on both the free catalyst and the catalyst-substrate complex.
The M20 mechanism consists of \(6\) chemical species,
\[
    \{\sA, \sP, \scat, \scatA, \scatI, \scatAI\},
\]
and includes both reversible and irreversible reactions. As in M1, the reactions
\[
    \sA + \scat \;\rightleftarrows\; \scatA
    \quad\text{and}\quad
    \scatA \;\rightleftarrows\; \sP + \scat
\]
form a closed subnetwork, while the additional reactions
\[
    \scat \xrightarrow{k_I} \scatI,
    \qquad
    \scatA \xrightarrow{k_{AI}} \scatAI
\]
introduce irreversible loss channels for the catalyst. This structure results in a partially open network, increasing the dimensionality of the coefficient matrix and posing a more challenging recovery problem. 
The graph of M20 is shown in Figure~\ref{fig:M20-Graph}.

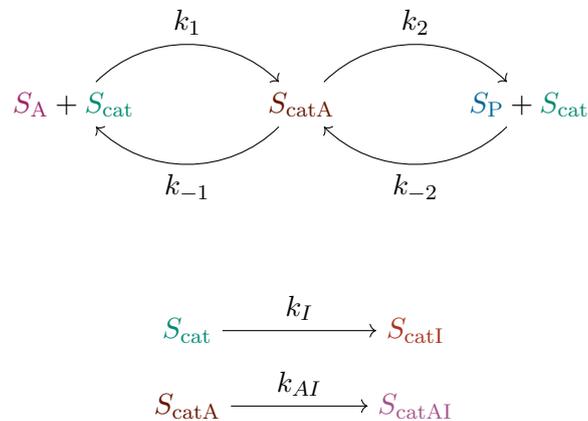
\begin{figure}[H]
    \centering
    \begin{tikzpicture}
        \node (A) at (0,0) {$\sA+\scat$};
        \node (B) at (3,0) {$\scatA$};
        \node (C) at (6,0) {$\sP+\scat$};
        
        \draw[->] (A) to[out=45,in=135]  node[above] {$k_{1}$} (B);
        \draw[<-] (A) to[out=-45,in=-135]  node[below] {$k_{-1}$} (B);
        
        \draw[->] (B) to[out=45,in=135]   node[above] {$k_{2}$} (C);
        \draw[<-] (B) to[out=-45,in=-135]   node[below] {$k_{-2}$} (C);

        \node (D) at (1.5,-3) {$\scat$};
        \node (E) at (4.5,-3) {$\scatI$};
        
        \node (F) at (1.5,-4) {$\scatA$};
        \node (G) at (4.5,-4) {$\scatAI$};
        
        \draw[->] (D) to node[above] {$k_{I}$} (E);
        \draw[->] (F) to node[above] {$k_{AI}$} (G);
    \end{tikzpicture}
    \caption{Graph of the M20 Model}
    \label{fig:M20-Graph}
\end{figure}

As before, accurate recovery requires the dictionary matrix formed from multiple experiments to have full row rank. Due to the increased number of reactions and species, we found numerically that stacking \(w = 8\) experiments with distinct initial conditions was sufficient to reliably satisfy this condition.

\subsubsection{Model recovery analysis}
For the numerical experiments associated with Mechanism~M20, the training data matrix \(\mathbf{X}\) is constructed as follows. The kinetic constants
\[
    \{k_1, k_2, k_{-1}, k_{-2}, k_I, k_{AI}\}
\]
are sampled independently and uniformly from the interval \([5 \times 10^{-2}, 1.0]\), defining the ground-truth coefficient matrix \(\mathbf{C}_{\mathrm{ex}}\). We then generate \(w = 8\) initial condition vectors
\[
    \mathbf{x}^{(i)}(0)
    = \left(\xA^{(i)}(0), \xP^{(i)}(0), \xcat^{(i)}(0), \xcatA^{(i)}(0), \xcatI^{(i)}(0), \xcatAI^{(i)}(0)\right),
\]
for \(i = 1, \dots, w\), with each component drawn independently from a uniform distribution on \([0,1]\). Using these initial conditions, the ODE system defined by \(\mathbf{C}_{\mathrm{ex}}\) is integrated over the interval \([0,20]\), discretised into uniformly spaced time points.
We examine the behaviour of the reconstruction errors
\begin{align*}
    \|\Delta\mathbf{C}_{\mathrm{int,stls}}\|_{2}, \quad
    \|\Delta\mathbf{C}_{\mathrm{int,ls}}\|_{2}, \quad
    \|\Delta\mathbf{C}_{\mathrm{dif,stls}}\|_{2}, \quad
    \|\Delta\mathbf{C}_{\mathrm{dif,ls}}\|_{2},
\end{align*}
as the number of time points used in the simulation increases from \(50\) to \(1000\), in increments of~\(50\).

\begin{figure}[h]
    \centering
    \includegraphics[width=\textwidth]{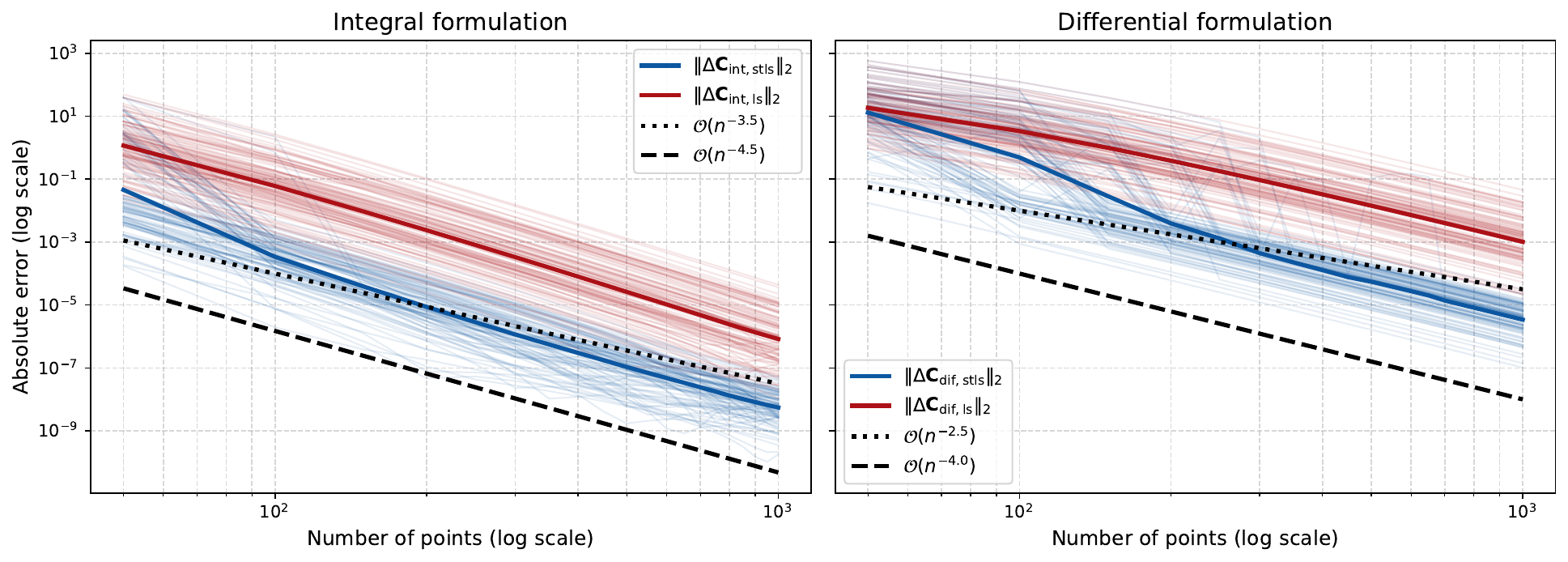}
    \caption{
    M20 reconstruction error for integration-based and differentiation-based recovery methods across increasing numbers of time points. Each faint line corresponds to one of 100 independent trials; bold lines represent the geometric mean of all realisations. In each subplot, we also include two dashed reference lines to assess the decay rate with respect to the number of time points. One line follows the theoretical error bound from Theorem~\ref{thm:2} and the other was fitted to the observed numerical decay rate.}
    \label{fig:M20_errbnd_noiseoff}
\end{figure}

This experiment is repeated for 100 independent realisations, each with newly sampled kinetic constants and initial conditions. Figure~\ref{fig:M20_errbnd_noiseoff} summarises the resulting reconstruction errors. As in the M1 case, the integration-based formulation consistently achieves lower reconstruction error than the differential formulation. While the increased complexity of the M20 network leads to larger absolute errors overall, the observed decay rates remain comparable to, and in some cases exceed, the theoretical predictions.

We further evaluate the structural accuracy of the recovered coefficient matrices using the support mismatch metric, defined as the total number of false positives and false negatives relative to the ground-truth support of \(\mathbf{C}_{\mathrm{ex}}\).
To estimate this metric statistically, we perform 1000 independent trials at four temporal resolutions: 25, 50, 75, and 100 uniformly spaced time points. For each trial, synthetic data is generated using the procedure described above, and the matrices \(\mathbf{C}_{\mathrm{int,stls}}\) and \(\mathbf{C}_{\mathrm{dif,stls}}\) are recovered. The support mismatch relative to \(\mathbf{C}_{\mathrm{ex}}\) is then computed.
The distributions of support mismatches are shown in Figure~\ref{fig:M20_support-mismatch}. As expected, the more complex structure of M20 leads to broader distributions than in M1, particularly for the differentiation-based formulation. Nevertheless, the integration-based approach continues to demonstrate superior structural recovery, with a substantial fraction of trials achieving near-perfect support identification even at relatively coarse temporal resolutions.

\begin{figure}[H]
    \centering
    \includegraphics[width=0.90\textwidth]{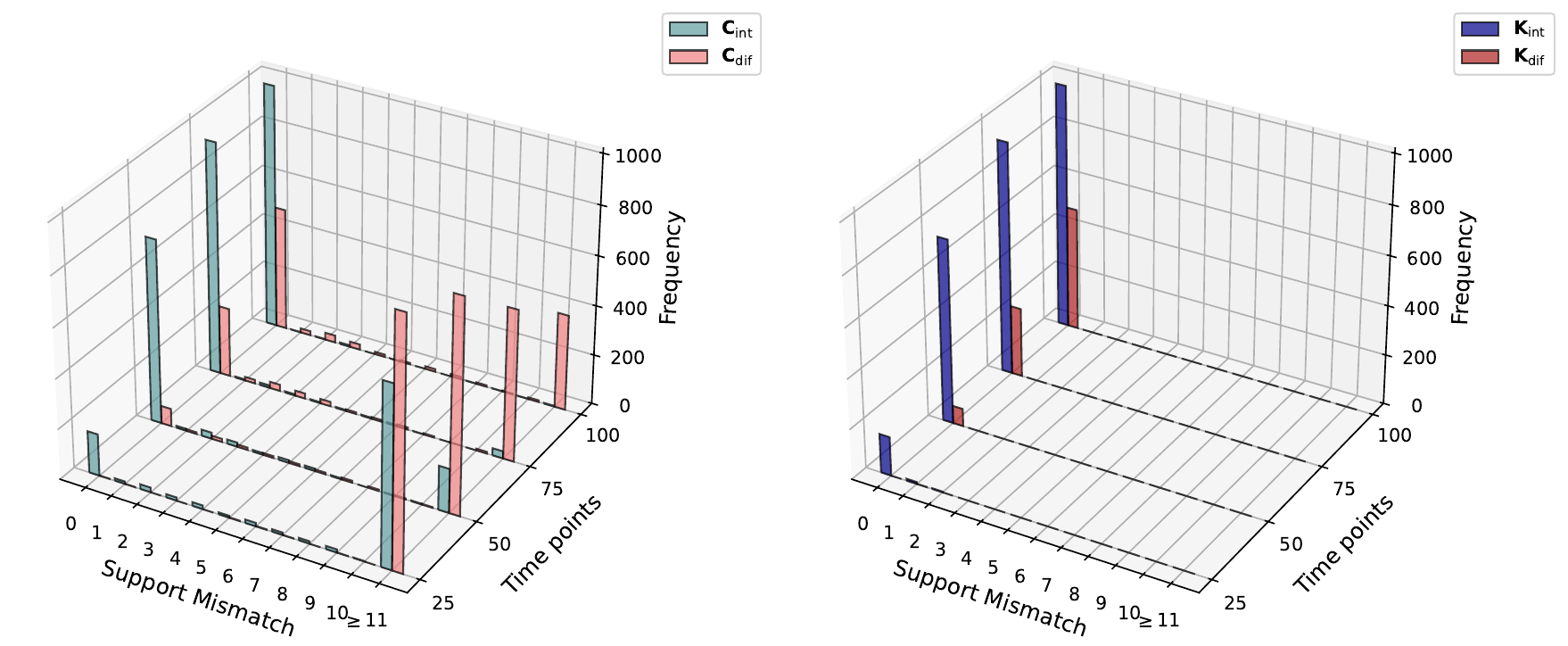}
    \caption{
    Left: M20 support mismatch between the ground-truth coefficient matrix and recovered matrices over 1000 trials, shown for multiple temporal resolutions and for both integration- and differentiation-based formulations. Right: Quality of the recovered Kirchhoff matrices.}
    \label{fig:M20_support-mismatch}
\end{figure}

\subsubsection{Graph recovery analysis}
To illustrate graph recovery for Model~M20, we consider a single representative instance of the algorithm. We fix a set of kinetic constants sampled uniformly from \([5 \times 10^{-2}, 1.0]\) and generate a collection of \(w=8\) initial conditions as described above. Using this fixed system, we integrate the ODEs over the interval \([0,20]\) with \(50\) uniformly spaced time points.

Applying the integration-based recovery procedure to the resulting data yields the recovered CRN graph shown in Figure~\ref{fig:M20_graph-comparison_1}, alongside the ground-truth network. We observe that the recovery procedure yields an incorrect graph. This behaviour stems from the open nature of the M20 model: the ground-truth coefficient matrix \(\mathbf{C}\) contains inactive columns corresponding to the complexes \(\scatI\) and \(\scatAI\). As a result, regardless of the accuracy of the recovered matrix \(\mathbf{C}_{\mathrm{rec}}\), these complexes cannot be identified as active using SINDy-based recovery alone.
\begin{figure}[H]
   \hspace*{-8mm}\includegraphics[width=1.1\textwidth]{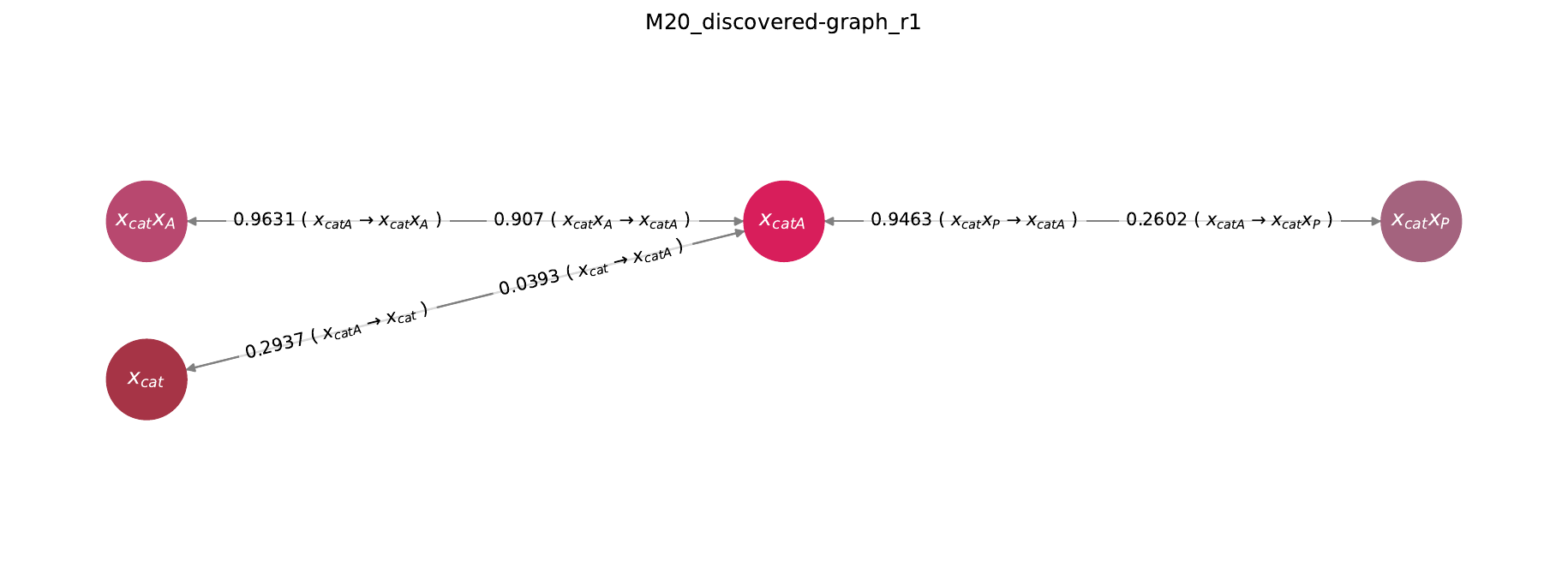}
   \hspace*{-8mm}\includegraphics[width=1.1\textwidth]{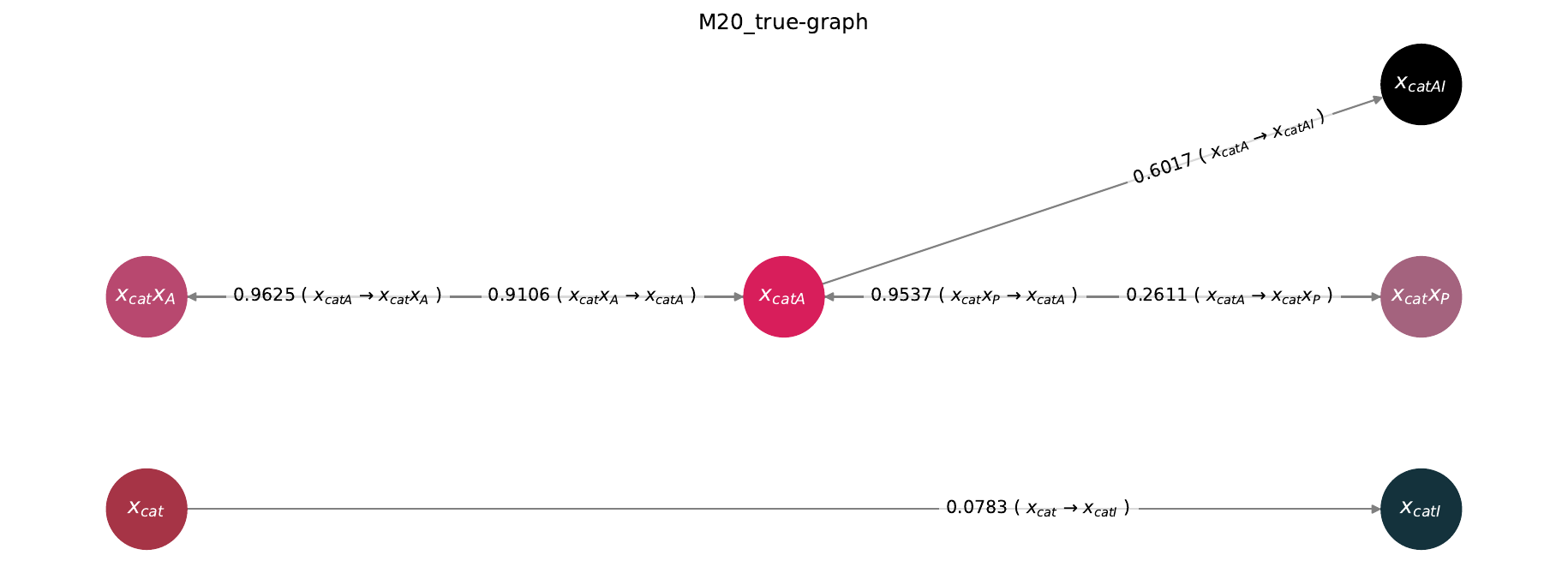}
    \caption{
    Top: Incorrectly recovered CRN graph for M20 obtained using the integration-based formulation with 50 time points. 
    Bottom: Ground-truth CRN graph for the M20 system.}
    \label{fig:M20_graph-comparison_1}
\end{figure}

This limitation can be mitigated by explicitly accounting for mass loss through the introduction of a synthetic zero-complex, which aggregates all irreversible outflows from the system. In practice, this is achieved by appending an additional zero column to the matrix \(\mathbf{Q}\) obtained during the reduction of \(\mathbf{C}_{\mathrm{rec}}\). The resulting recovered graph, incorporating the zero-complex, is shown in Figure~\ref{fig:M20_graph-comparison_2}, where we can see that the zero-complex effectively captures the reactions ending at \(\scatI\) and \(\scatAI\), thus showing that recovery is possible for open systems by adding this zero-complex. 

\begin{figure}[H]
   \hspace*{-8mm}\includegraphics[width=1.1\textwidth]{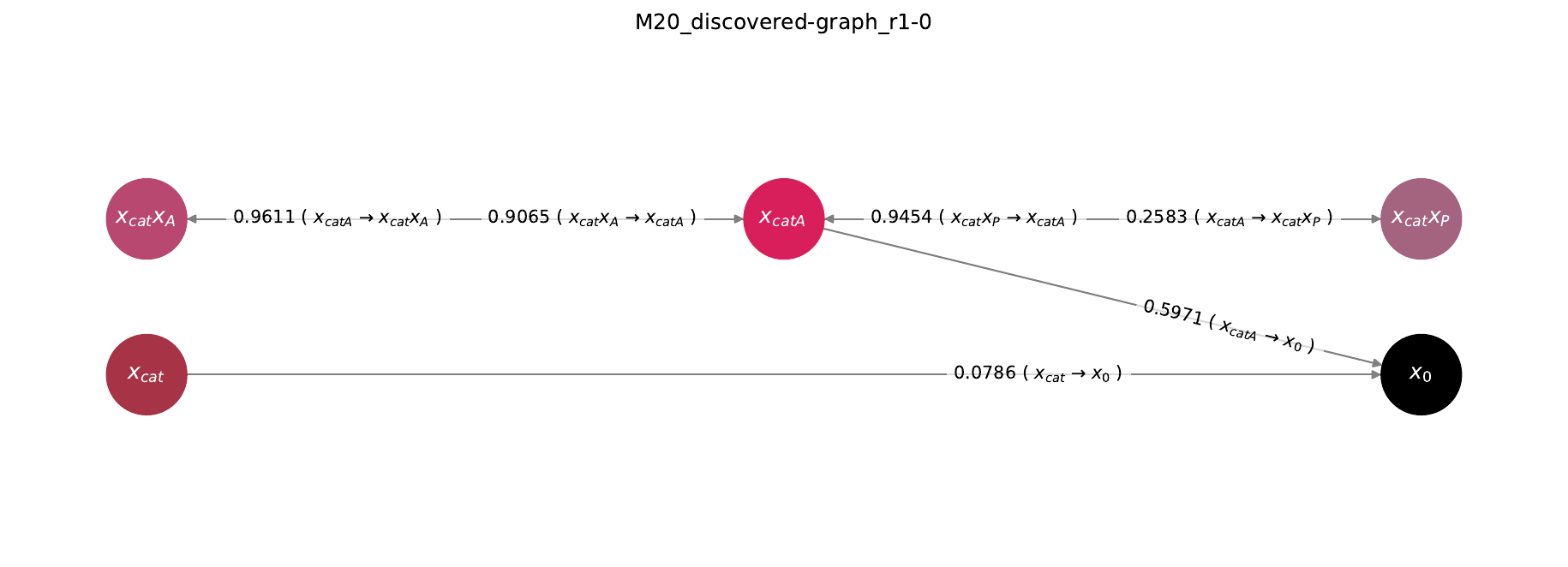}
    \caption{Correctly recovered CRN graph for M20 obtained using the integration-based formulation with 50 time points, after adding a zero-complex \(\mathbf{Q}\).}
    \label{fig:M20_graph-comparison_2}
\end{figure}

For small networks, a different approach to constructing the \(\mathbf{Q}\) matrix for graph recovery is to assume that every species we know of acts as a single-order reactant complex, and just eliminate the rows associated to second-order inactive rows of \(\mathbf{C}_{\mathrm{rec}}\). In this setting, the recovery algorithm correctly reconstructs the full M20 reaction graph, including the irreversible loss pathways. 
The resulting graph is shown in Figure~\ref{fig:M20_graph-comparison_3}, and confirms that the observed structural discrepancy arises solely from the unidentifiability of inactive complexes in the data-driven setting, rather than from a failure of the graph recovery procedure itself. 

\begin{figure}[H]
   \hspace*{-8mm}\includegraphics[width=1.1\textwidth]{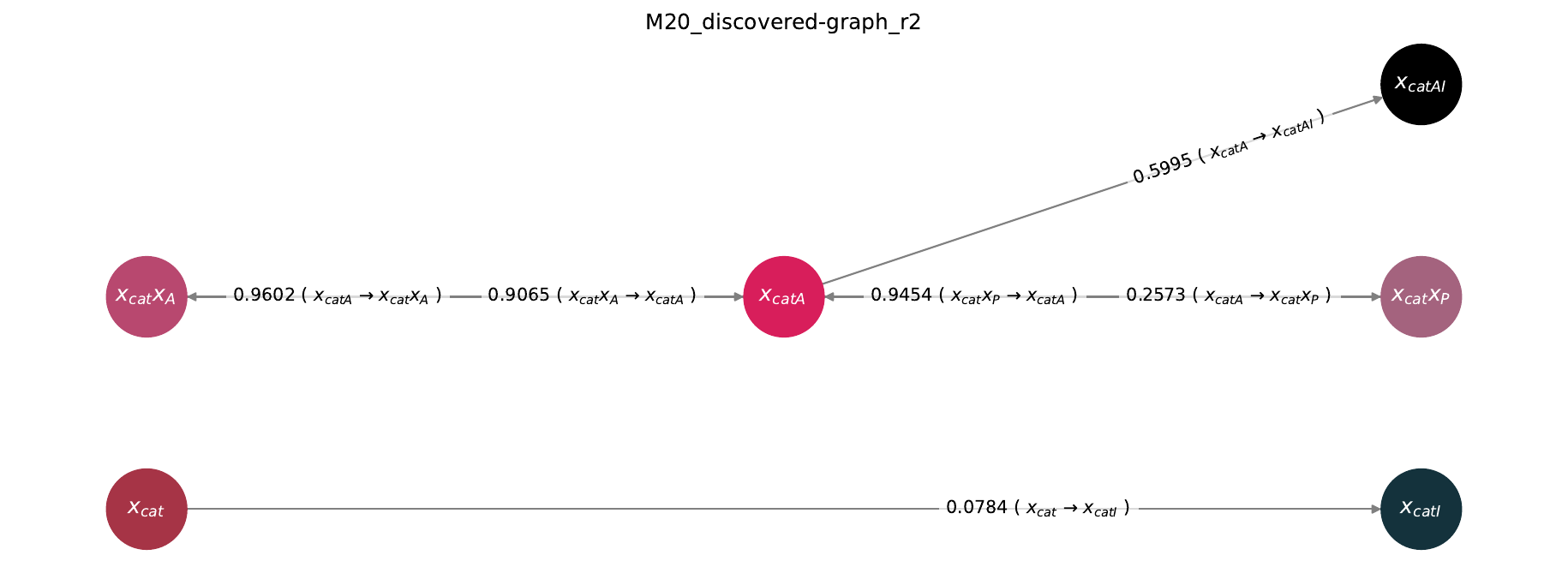}
    \caption{Recovered CRN graph for M20 obtained using the integration-based formulation with 50 time points,  treating all species as single-order active complexes.}
    \label{fig:M20_graph-comparison_3}
\end{figure}

\subsubsection{Noisy measurements of the M20 mechanism}
Finally, we assess the robustness of the recovery procedure for Mechanism~M20 in the presence of measurement noise. As in the M1 experiments, each entry of the data matrix \(\mathbf{X}\) is independently perturbed by zero-mean Gaussian noise with variance \(10^{-4}\). We then examine the resulting reconstruction errors
\[
\|\Delta\mathbf{C}_{\mathrm{int,stls}}\|_{2},
\quad
\|\Delta\mathbf{C}_{\mathrm{dif,stls}}\|_{2}.
\]
The results, shown in Figure~\ref{fig:M20_errbnd_noiseon}, confirm that the integration-based formulation remains markedly more robust to noise than the differential approach. Despite the higher dimensionality and partial openness of the M20 network, the integration-based method consistently achieves lower reconstruction error and improved stability, reinforcing the conclusions drawn from the M1 experiments. 

We note that some of the STLS curves in Figure~\ref{fig:M20_errbnd_noiseon} show a zig-zag type behaviour. This seems to appear in cases where STLS stops too early, leaving two many nonzero coefficients in the least squares problem, and thereby reducing the recovery performance close to that of the unregularized (full) least squares algorithm. This is testament to the inherent difficulty of sparse recovery and it is possible that other algorithms for sparse recovery (like LASSO, OMP) might perform better. Overall, however, STLS is doing a good job given its simplicity, with the correct coefficient matrices recovered in a majority of trials, and on average better than unregularized least squares (as indicated by the lower geometric mean curves in the plots).

\begin{figure}[h!]
    \centering
    \includegraphics[width=\textwidth]{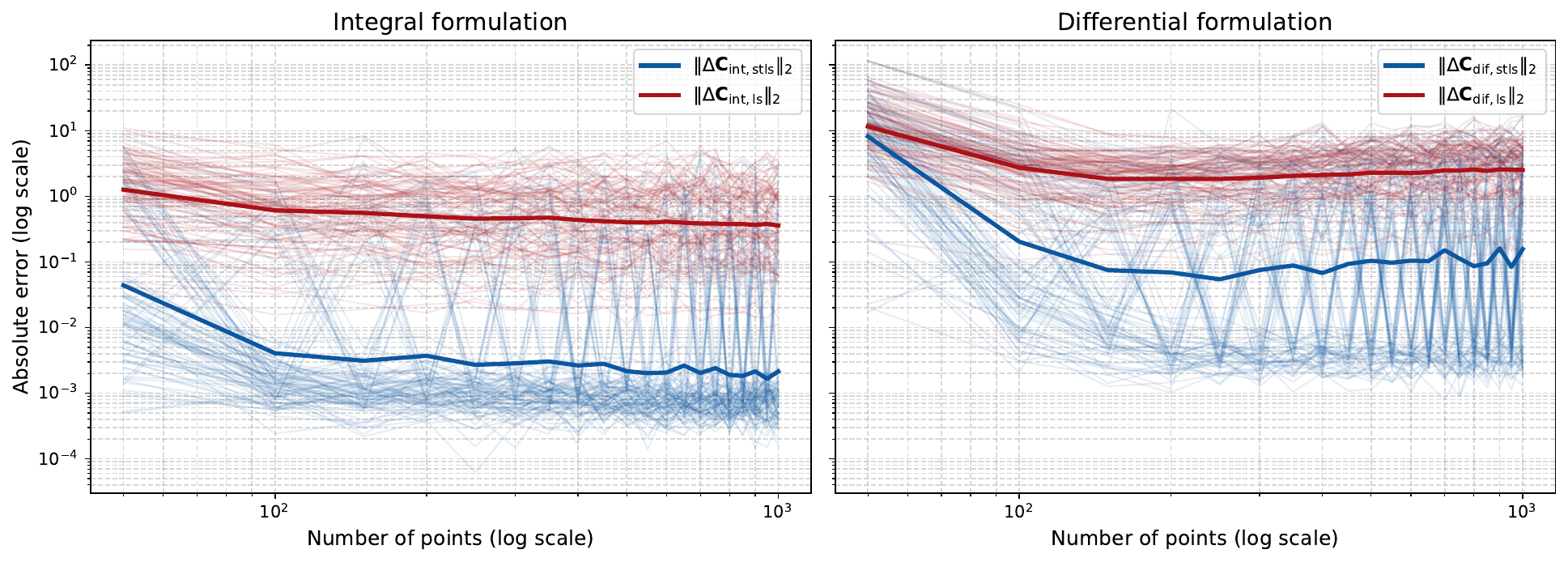}
    \caption{
    Noisy M20 reconstruction error for integration-based and differentiation-based recovery methods across increasing numbers of time points. Each faint line corresponds to one of 100 independent trials; bold lines represent the geometric mean of all realisations.}
    \label{fig:M20_errbnd_noiseon}
\end{figure}

\subsection{Van de Vusse reaction}

We conclude our numerical study with the Van de Vusse reaction, a classical benchmark model in chemical reaction engineering featuring both linear and nonlinear reaction pathways. The reaction mechanism involves \(4\) chemical species $\{x_{1}, x_{2}, x_{3}, x_{4}\}$ 
participating in a sequential conversion and a parallel bimolecular side reaction given by
\[
    2x_{1} \xrightarrow{k_1} x_{2},
    \qquad
    x_{1} \xrightarrow{k_2} x_{3} \xrightarrow{k_3} x_{4}.
\]
This defines an open reaction network due to the irreversible conversion of reactants into final products. 

\subsubsection{Model recovery analysis}
For the numerical experiments associated with the Van de Vusse reaction, the training data matrix~\(\mathbf{X}\) is generated as follows. The kinetic constants are picked from \cite{burnham2007identifying} as
\begin{equation} \label{eq:vdv_ks}
    k_{1} = 10^{-3}, \quad k_{2} = 6.85\times10^{-3}, \quad k_{3} = 2.48\times10^{-3},
\end{equation}
defining the ground-truth coefficient matrix \(\mathbf{C}_{\mathrm{ex}}\). We generate \(w = 4\) initial condition vectors
\[
    \mathbf{x}^{(i)}(0)
    = \left(x_{1}^{(i)}(0), x_{2}^{(i)}(0), x_{3}^{(i)}(0), x_{4}^{(i)}(0)\right),
    \qquad i = 1, \dots, w,
\]
with each component drawn independently from a uniform distribution on \([0,1]\). For each initial condition, the resulting ODE system is integrated over the time interval \([0,20]\), discretised into uniformly spaced time points. 
We analyse the reconstruction errors
\begin{align*}
    \|\Delta\mathbf{C}_{\mathrm{int,stls}}\|_{2}, \quad
    \|\Delta\mathbf{C}_{\mathrm{int,ls}}\|_{2}, \quad
    \|\Delta\mathbf{C}_{\mathrm{dif,stls}}\|_{2}, \quad
    \|\Delta\mathbf{C}_{\mathrm{dif,ls}}\|_{2},
\end{align*}
as the number of time points used in the simulation is increased from \(50\) to \(1000\), in increments of \(50\). The experiment is repeated for 100 independent trials, each with newly sampled kinetic constants and initial conditions. Figure~\ref{fig:vdv_errbnd_noiseoff} summarises the resulting reconstruction errors. As observed for Models~M1 and~M20, the integration-based formulation consistently achieves lower reconstruction error and faster decay rate than the differentiation-based approach. The presence of nonlinear reaction terms leads to moderately increased reconstruction error, particularly for coarse temporal resolutions.

\begin{figure}[h]
    \centering
    \includegraphics[width=\textwidth]{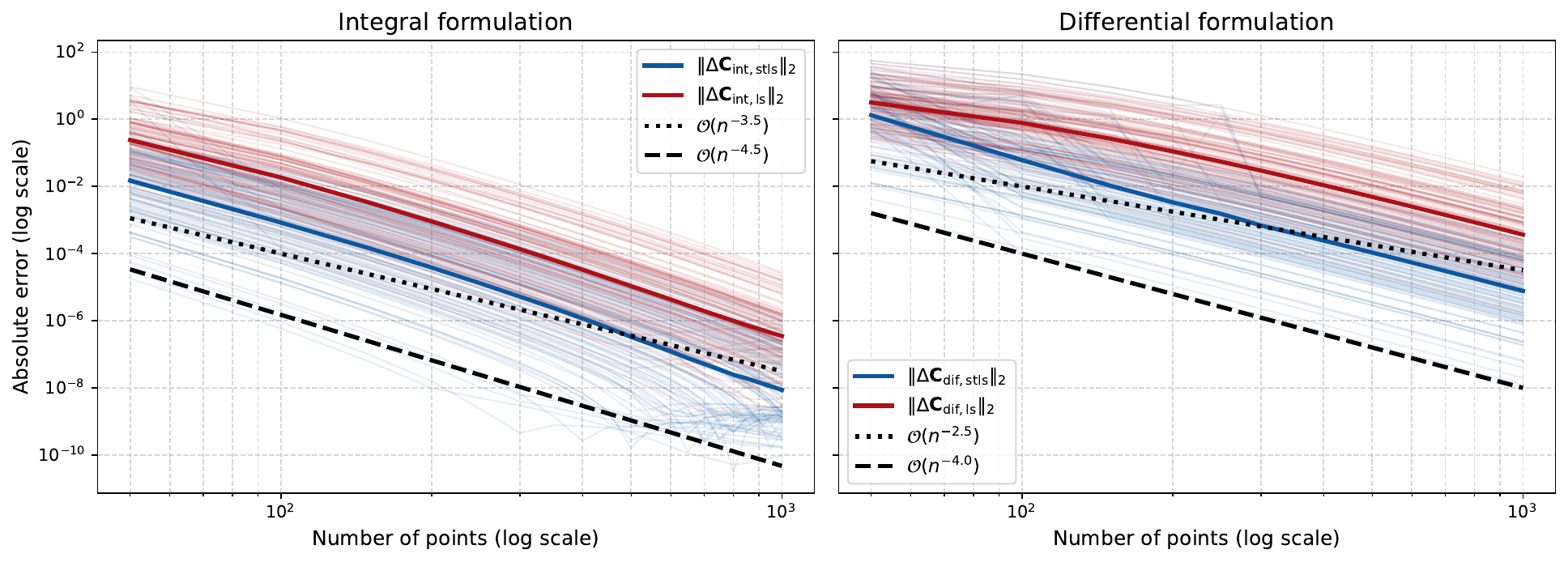}
    \caption{
    Van de Vusse reconstruction error for integration-based and differentiation-based recovery methods across increasing numbers of time points. Each faint line corresponds to one of 100 independent trials; bold lines represent the geometric mean of all realisations. In each subplot, we also include two dashed reference lines to assess the decay rate with respect to the number of time points. One line follows the theoretical error bound from Theorem~\ref{thm:2} and the other was fitted to the observed numerical decay rate.}
    \label{fig:vdv_errbnd_noiseoff}
\end{figure}

We next evaluate the structural accuracy of the recovered coefficient matrices by measuring their support mismatch with respect to the ground-truth matrix \(\mathbf{C}_{\mathrm{ex}}\).  
We perform 1000 independent trials at four temporal resolutions: 25, 50, 75, and 100 uniformly spaced time points. For each trial, synthetic data is generated as described above, and the matrices \(\mathbf{C}_{\mathrm{int,stls}}\) and \(\mathbf{C}_{\mathrm{dif,stls}}\) are computed. The resulting support mismatches are shown in Figure~\ref{fig:vdv_support-mismatch}. As expected, the integration-based formulation yields significantly more accurate support recovery, particularly at lower temporal resolutions.

\begin{figure}[H]
    \centering
    \includegraphics[width=0.9\textwidth]{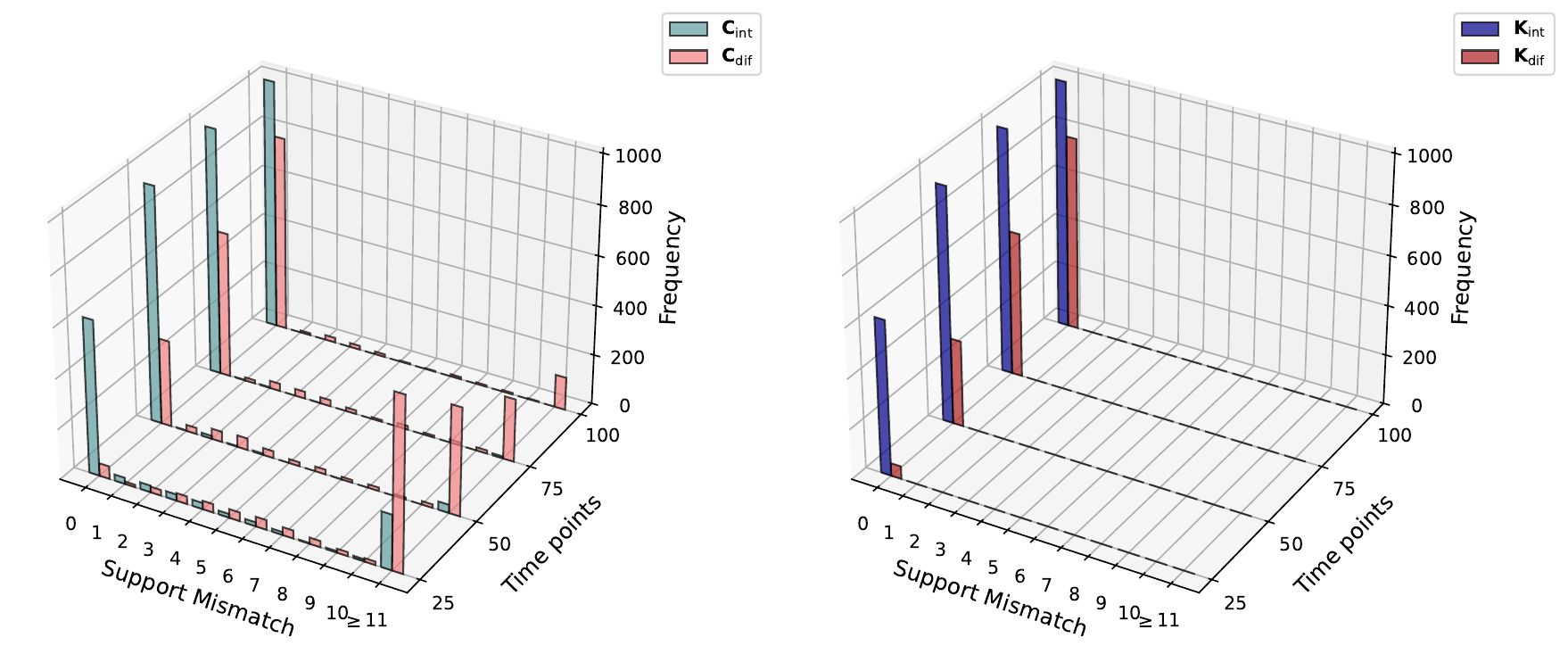}
    \caption{
    Left: Support mismatch between the ground-truth coefficient matrix and recovered matrices for the Van de Vusse reaction over 1000 independent trials, shown for multiple temporal resolutions and both integration- and differentiation-based formulations. Right: Quality of the recovered Kirchhoff matrices.}
    \label{fig:vdv_support-mismatch}
\end{figure}

\subsubsection{Graph recovery analysis}

We illustrate the graph recovery performance of our method on a representative instance of the Van de Vusse reaction. We use our fixed set of kinetic constants~Eq.(\ref{eq:vdv_ks}) and \(w=4\) initial conditions are generated as described above. The system is integrated over the interval \([0,20]\) using \(50\) uniformly spaced time points.

We first attempt the recovery using the complex stochiometry matrix \(\mathbf{Q}\) constructed by removing all columns corresponding to inactive columns of the recovered coefficient matrix. The result is shown on Figure~\ref{fig:vdv_graph-comparison_1}. As was the case for the M20 model, for open networks this produces an incorrect graph.

\begin{figure}[H]
   \hspace*{-8mm}\includegraphics[width=1.1\textwidth]{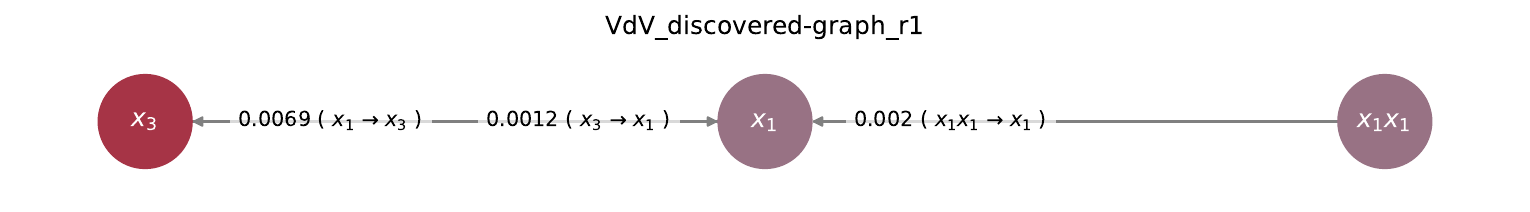}
   \hspace*{-8mm}\includegraphics[width=1.1\textwidth]{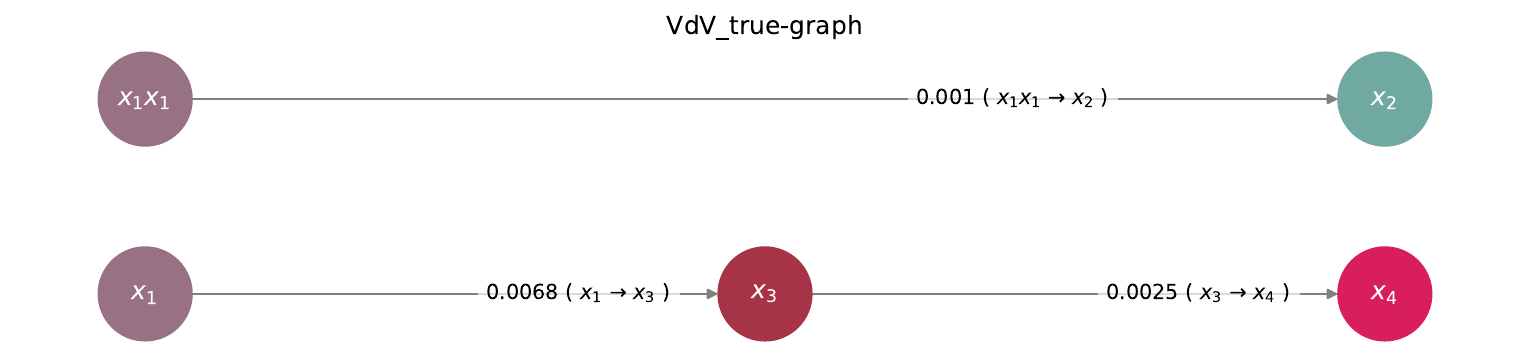}
    \caption{
    Top: Incorrectly recovered CRN graph obtained using the integration-based formulation with 50 time points. 
    Bottom: Ground-truth CRN graph for the Van de Vusse reaction.
    }
    \label{fig:vdv_graph-comparison_1}
\end{figure}

We next perform graph recovery by introducing a zero-complex in addition to the active complexes. The resulting recovered network is shown in Figure~\ref{fig:vdv_graph-comparison_2}. This reconstruction illustrates the phenomenon of chemical equivalence \cite{szederkenyi2011inference}, as the recovered network reproduces the same equations of motion for the complexes \(x_{1}\) and \(x_{3}\) as in the original Van de Vusse system, but the overall recovered structure is different to the original Van de Vusse reaction graph. For this model, we require alternative filtration schemes in constructing the matrices \(\mathbf{C}_{\mathrm{eff}}\) and \(\mathbf{Q}_{\mathrm{eff}}\) for the graph recovery problem. In Figure~\ref{fig:vdv_graph-comparison_2}, we show the recovered structure when we reduce \(\mathbf{C}_{\mathrm{eff}}\) by removing only inactive columns corresponding to quadratic complexes. In this case, we recover the correct structure of the Van de Vusse model.

\begin{figure}[H]
   \hspace*{-8mm}\includegraphics[width=1.1\textwidth]{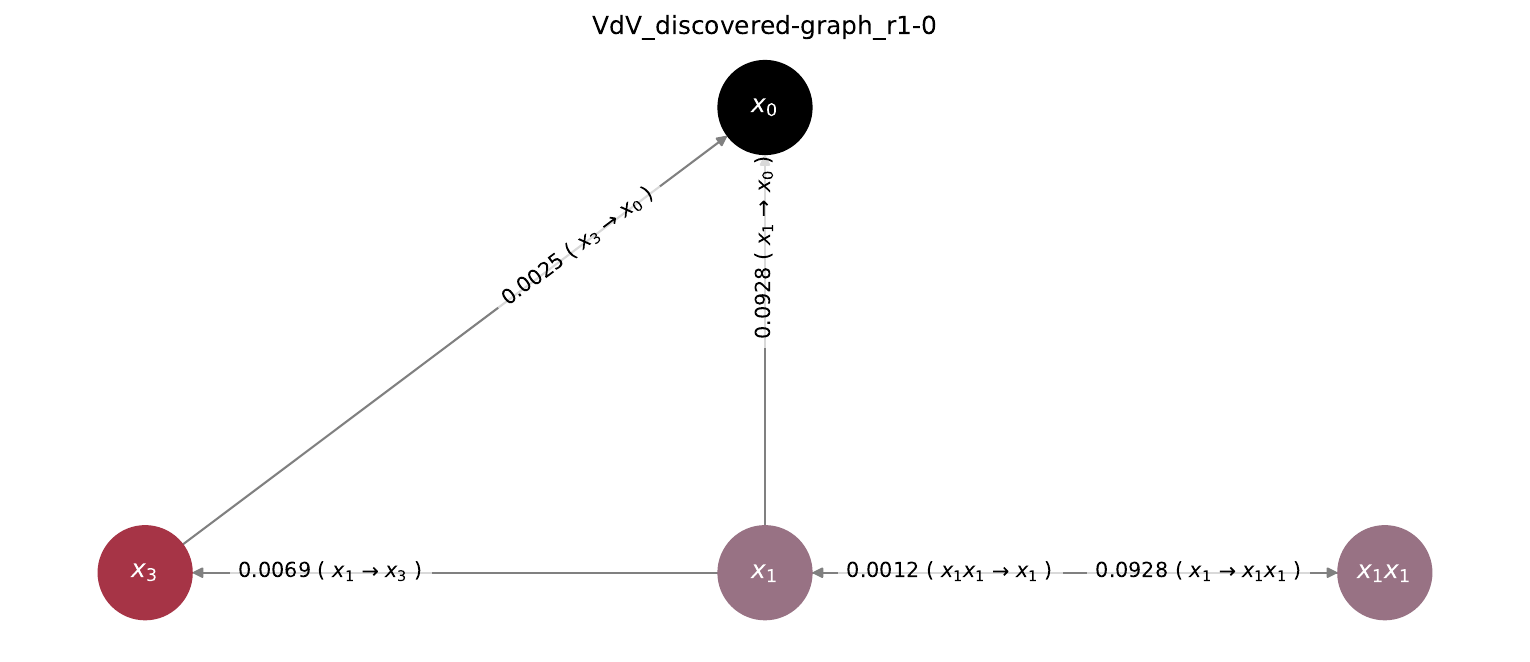}
   \hspace*{-8mm}\includegraphics[width=1.1\textwidth]{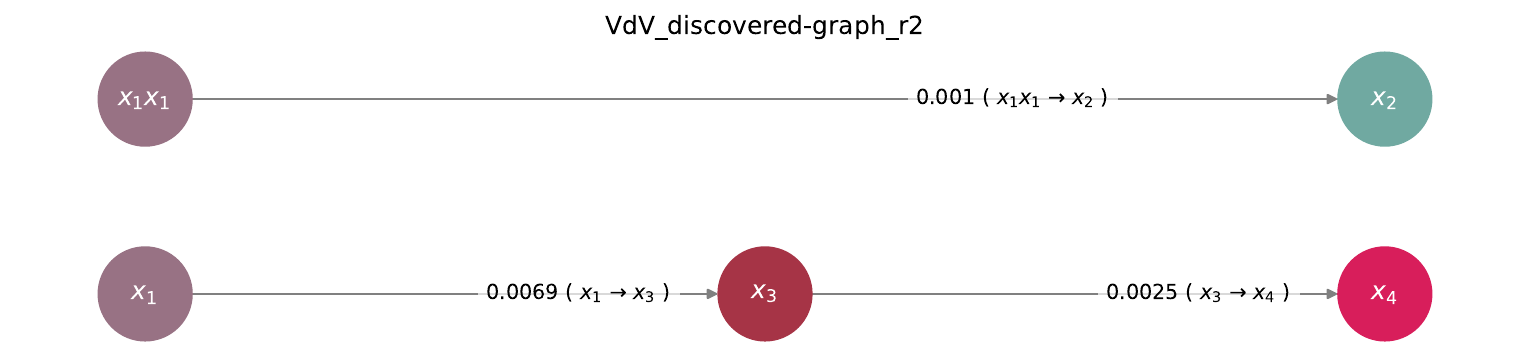}
    \caption{Top: Incorrectly recovered CRN graph obtained using the integration-based formulation with 50 time points after adding a zero-complex. Bottom: Correctly recovered CRN graph obtained using the integration-based formulation with 50 time points and assuming all species as reactant complexes.}
    \label{fig:vdv_graph-comparison_2}
\end{figure}

\subsubsection{Noisy Van de Vusse reaction}

Finally, we assess the robustness of the recovery procedure in the presence of measurement noise. Each entry of the data matrix \(\mathbf{X}\) is independently perturbed by additive Gaussian noise with zero mean and variance \(10^{-4}\). In Figure~\ref{fig:vdv_errbnd_noiseon} we plot the resulting reconstruction errors
\[
    \|\Delta\mathbf{C}_{\mathrm{int,stls}}\|_{2},
    \quad
    \|\Delta\mathbf{C}_{\mathrm{dif,stls}}\|_{2}.
\]

As before, these results demonstrate that the integration-based formulation remains robust and accurate for reaction networks involving nonlinear interactions, such as the Van de Vusse reaction, and further highlight its advantages over differentiation-based approaches in both noise-free and noisy settings.

\begin{figure}[h]
    \centering
    \includegraphics[width=\textwidth]{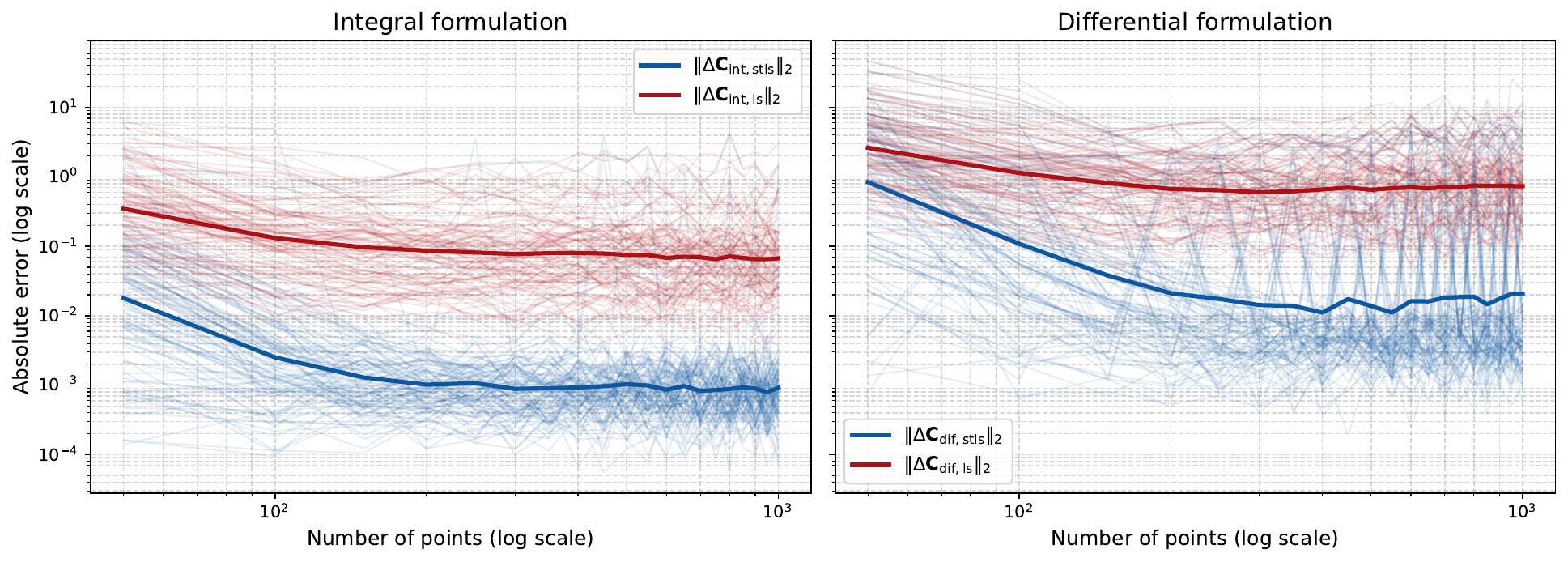}
    \caption{
    Noisy Van de Vusse reconstruction error for integration-based and differentiation-based recovery methods across increasing numbers of time points. Each faint line corresponds to one of 100 independent trials; bold lines represent the geometric mean of all realisations.}
    \label{fig:vdv_errbnd_noiseon}
\end{figure}

\section{Conclusions}\label{sec:concl}

We have proposed a new framework for the full mechanistic reconstruction of chemical reaction networks (CRNs) from concentration data. Theoretical analysis and numerical evidence suggests that recovering differential equations in integral form is superior to using numerical differentiation in terms of accuracy and robustness to noise. We have also introduced a automatic procedure to recover admissible mass-action mechanisms from the equations. 

In future work we hope to extend our framework to the case of incomplete measurements, where concentrations of only a subset of the species are available. This is particularly challenging in cases where not even the number of hidden species is known. Also, in our current implementation we assume that all experiments happen on similar time-scales as we keep the simulation time window and the  integration matrix fixed across simulations. There is no intrinsic reason our framework could not be adapted to the case of  varying time-scales. It would mainly require the use of scale-adapted integration matrices.  

\paragraph{Acknowledgments.}  S.\,G.\ acknowledges
funding from the UK’s Engineering and Physical Sciences Research Council (EPSRC grant EP/Z533786/1) and the Royal Society (RS Industry Fellowship IF/R1/231032).  I.\,L.\ acknowledges funding from the European Research Council (Advanced Grant RuCat 833337). We are grateful to Matthew Colbrook and Stefan Klus for useful discussions on the topic of this paper.

\bibliographystyle{siam}
\bibliography{bibliography}

@book{Espenson1995,
  author    = {James H. Espenson},
  title     = {Chemical Kinetics and Reaction Mechanisms},
  edition   = {2nd},
  publisher = {McGraw–Hill},
  year      = {1995}
}

@book{scherzer,
  title={{Variational Methods in Imaging}},
  author={Scherzer, Otmar and Grasmair, Markus and Grossauer, Harald and Haltmeier, Markus and Lenzen, Frank},
  volume={167},
  year={2009},
  publisher={Springer}
}

@article{blackmond2005reaction,
  title={Reaction progress kinetic analysis: a powerful methodology for mechanistic studies of complex catalytic reactions},
  author={Blackmond, Donna G},
  journal={Angewandte Chemie International Edition},
  volume={44},
  number={28},
  pages={4302--4320},
  year={2005},
  publisher={Wiley Online Library}
}

@article{bures2016variable,
  title={Variable time normalization analysis: general graphical elucidation of reaction orders from concentration profiles},
  author={Bur{\'e}s, Jordi},
  journal={Angewandte Chemie},
  volume={128},
  number={52},
  pages={16318--16321},
  year={2016},
  publisher={Wiley Online Library}
}

@article{bhore1990delplot,
  title={The delplot technique: a new method for reaction pathway analysis},
  author={Bhore, Nazeer A and Klein, Michael T and Bischoff, Kenneth B},
  journal={Industrial \& Engineering Chemistry Research},
  volume={29},
  number={2},
  pages={313--316},
  year={1990},
  publisher={ACS Publications}
}

@article{brunton2016discovering,
  title={Discovering governing equations from data by sparse identification of nonlinear dynamical systems},
  author={Brunton, Steven L and Proctor, Joshua L and Kutz, J Nathan},
  journal={Proceedings of the National Academy of Sciences},
  volume={113},
  number={15},
  pages={3932--3937},
  year={2016},
  publisher={National Acad Sciences}
}

@article{schaeffer2017sparse,
  title={Sparse model selection via integral terms},
  author={Schaeffer, Hayden and McCalla, Scott G},
  journal={Physical Review E},
  volume={96},
  number={2},
  pages={023302},
  year={2017},
  publisher={APS}
}

@article{forootani2023robust,
  title={A robust {SINDy} approach by combining neural networks and an integral form},
  author={Forootani, Ali and Goyal, Pawan and Benner, Peter},
  journal={arXiv preprint arXiv:2309.07193},
  year={2023}
}

@article{zhang2019convergence,
  title={On the convergence of the {SINDy} algorithm},
  author={Zhang, Linan and Schaeffer, Hayden},
  journal={Multiscale Modeling \& Simulation},
  volume={17},
  number={3},
  pages={948--972},
  year={2019},
  publisher={SIAM}
}

@inproceedings{bhatt2023sindy,
  title={{SINDy-CRN: Sparse Identification of Chemical Reaction Networks from Data}},
  author={Bhatt, Nirav and Jayawardhana, Bayu and Plaza, Santiago S{\'a}nchez-Escalonilla},
  booktitle={2023 62nd IEEE Conference on Decision and Control (CDC)},
  pages={3512--3518},
  year={2023},
  organization={IEEE}
}

@article{burnham2007identifying,
  title={Identifying chemical reaction network models},
  author={Burnham, SC and Willis, MJ and Wright, AR},
  journal={IFAC Proceedings Volumes},
  volume={40},
  number={5},
  pages={225--230},
  year={2007},
  publisher={Elsevier}
}

@article{mangan2016inferring,
  title={Inferring biological networks by sparse identification of nonlinear dynamics},
  author={Mangan, Niall M and Brunton, Steven L and Proctor, Joshua L and Kutz, J Nathan},
  journal={IEEE Transactions on Molecular, Biological, and Multi-Scale Communications},
  volume={2},
  number={1},
  pages={52--63},
  year={2016},
  publisher={IEEE}
}

@article{szederkenyi2011inference,
  title={Inference of complex biological networks: distinguishability issues and optimization-based solutions},
  author={Szederk{\'e}nyi, G{\'a}bor and Banga, Julio R and Alonso, Antonio A},
  journal={BMC Systems Biology},
  volume={5},
  pages={1--15},
  year={2011},
  publisher={Springer}
}

@article{szederkenyi2010computing,
  title={Computing sparse and dense realizations of reaction kinetic systems},
  author={Szederk{\'e}nyi, Gabor},
  journal={Journal of Mathematical Chemistry},
  volume={47},
  number={2},
  pages={551--568},
  year={2010},
  publisher={Springer}
}

@book{feinberg2019foundations,
    author = {Feinberg, Martin},
    title = {Foundations of Chemical Reaction Network Theory},
    publisher = {Springer},
    year = {2019}
}

@article{HALL1976105,
title = {Optimal error bounds for cubic spline interpolation},
journal = {Journal of Approximation Theory},
volume = {16},
number = {2},
pages = {105-122},
year = {1976},
issn = {0021-9045},
author = {Charles A Hall and W.Weston Meyer}
}

@article{396bf6e1-ef54-3bf6-a49b-862db8404076,
 ISSN = {00361445, 10957200},
 URL = {http://www.jstor.org/stable/2030248},
 author = {G. W. Stewart},
 journal = {SIAM Review},
 number = {4},
 pages = {634--662},
 publisher = {Society for Industrial and Applied Mathematics},
 title = {On the Perturbation of Pseudo-Inverses, Projections and Linear Least Squares Problems},
 urldate = {2025-07-03},
 volume = {19},
 year = {1977}
}

@article{bures2023organic,
  title={Organic reaction mechanism classification using machine learning},
  author={Bur{\'e}s, Jordi and Larrosa, Igor},
  journal={Nature},
  volume={613},
  number={7945},
  pages={689--695},
  year={2023},
  publisher={Nature Publishing Group UK London}
}

@article{wei2022sparse,
  title={Sparse dynamical system identification with simultaneous structural parameters and initial condition estimation},
  author={Wei, Baolei},
  journal={Chaos, Solitons \& Fractals},
  volume={165},
  pages={112866},
  year={2022},
  publisher={Elsevier}
}

@article{hoffmann2019reactive,
  title={{Reactive SINDy: Discovering governing reactions from concentration data}},
  author={Hoffmann, Moritz and Fr{\"o}hner, Christoph and No{\'e}, Frank},
  journal={Journal of Chemical Physics},
  volume={150},
  number={2},
  year={2019},
  publisher={AIP Publishing}
}

@article{searson2007inference,
  title={Inference of chemical reaction networks using hybrid s-system models},
  author={Searson, Dominic P and Willis, Mark J and Horne, Simon J and Wright, Allen R},
  journal={Chemical Product and Process Modeling},
  volume={2},
  number={1},
  year={2007},
  publisher={De Gruyter}
}

@article{willis2016inference,
  title={Inference of chemical reaction networks using mixed integer linear programming},
  author={Willis, Mark J and von Stosch, Moritz},
  journal={Computers \& Chemical Engineering},
  volume={90},
  pages={31--43},
  year={2016},
  publisher={Elsevier}
}

@book{steinfeld1999chemical,
  title={Chemical Kinetics and Dynamics},
  author={Steinfeld, Jeffrey I and Francisco, Joseph Salvadore and Hase, William L},
  volume={2},
  year={1999},
  publisher={Prentice Hall Upper Saddle River, NJ}
}

@book{upadhyay2006chemical,
  title={Chemical Kinetics and Reaction Dynamics},
  author={Upadhyay, Santosh K},
  year={2006},
  publisher={Springer}
}

@article{zhang2019learning,
  title={Learning chemical reaction networks from trajectory data},
  author={Zhang, Wei and Klus, Stefan and Conrad, Tim and Sch{\"u}tte, Christof},
  journal={SIAM Journal on Applied Dynamical Systems},
  volume={18},
  number={4},
  pages={2000--2046},
  year={2019},
  publisher={SIAM}
}

\appendix

\section*{Appendix}

\begin{proof}[Proof of Theorem~\ref{thm:1}]
    \noindent \textbf{Differentiation error.}
    Let \( \boldsymbol{\Xi} = [\xi_{\alpha,i}] \) be the noise matrix, then
    \[
        \bar{\mathbf{X}} = \mathbf{X} + \boldsymbol{\Xi},
    \]
    and
    \begin{align*}
        \mathbf{E}_{\mathrm{dif}} = \dot{\mathbf{X}} - \bar{\mathbf{X}} \mathbf{L} = \underbrace{\dot{\mathbf{X}} - \mathbf{X} \mathbf{L}}_{\text{spline interpolation error}}  -\underbrace{\boldsymbol{\Xi}\mathbf{L}}_{\text{noise propagation}}
    \end{align*}
    The first term is bounded using classical cubic spline approximation theory \cite{HALL1976105}:
    \[
        \left| (\dot{\mathbf{X}} - \mathbf{X} \mathbf{L})_{\alpha,i} \right| \leq \frac{9 + \sqrt{3}}{216} \max_{t \in [t_0, t_n]} \left| x_\alpha^{(4)}(t) \right| h^3  = \frac{\kappa_{\mathrm{dif}}}{n^3}.
    \]
    For the noise propagation term, we have
    \[
        |(\boldsymbol{\Xi}\mathbf{L})_{\alpha,i}|
        = \left| \sum_{j=0}^{n} \xi_{\alpha,j} L_{j,i} \right|
        \le \varepsilon \sum_{j=0}^{n} |L_{j,i}|
        = \varepsilon \|\mathbf{L}_{:,i}\|_1.
    \]
    Combining both bounds yields
    \[
        \left|(\mathbf{E}_{\mathrm{dif}})_{\alpha,i}\right|
        \le \frac{\kappa_{\mathrm{dif}}}{n^3}
        + \varepsilon \|\mathbf{L}_{:,i}\|_1.
    \]

    \noindent\textbf{Integration error.}
    Let $\boldsymbol{\Delta}_\xi = \bar{\mathbf{D}} - \mathbf{D}$ denote the noise
    perturbation of the polynomial dictionary.  
    For a degree-$p$ monomial $d_{\beta,i} = \prod_{\alpha \in S_\beta} x_{\alpha,i}^{m_\alpha}$, 
    its noisy version is
    \[
        \bar{d}_{\beta,i} = \prod_{\alpha \in S_\beta} (x_{\alpha,i} + \xi_{\alpha,i})^{m_\alpha}.
    \]
    A first-order Taylor expansion in the $\xi_{\alpha,i}$ gives
    \[
        \bar{d}_{\beta,i} = d_{\beta,i} + \sum_{\alpha \in S_\beta} c_{\beta,\alpha,i} \, \xi_{\alpha,i} + \mathcal{O}(\varepsilon^2),
    \]
    where the coefficients $c_{\beta,\alpha,i}$ depend on the noiseless data, and are themselves polynomials of degree at most $p-1$ in the features $x_{\alpha,i}$.
    Hence, each entry of the dictionary perturbation satisfies
    \[
        |(\boldsymbol{\Delta}_\xi)_{\beta,i}| \le \varepsilon \sum_{\alpha \in S_\beta} |c_{\beta,\alpha,i}| = \varepsilon \ C_\beta, 
    \]
    where $C_\beta$ depends on the noiseless data. 
    Then, decomposing the integration error
    \[
        \mathbf{E}_{\mathrm{int}} = \int_{t_0}^{t} \mathbf{D}(s)\,ds - \bar{\mathbf{D}} \mathbf{J}
        = \underbrace{\int_{t_0}^{t} \mathbf{D}(s)\,ds - \mathbf{D}\mathbf{J}}_{\text{spline integration error}}
        - \underbrace{\boldsymbol{\Delta}_\xi \mathbf{J}}_{\text{noise propagation}},
    \]
    and using classical cubic spline approximation theory \cite{HALL1976105}, we obtain the entry-wise bound
    \[
        |(\mathbf{E}_{\mathrm{int}})_{\beta,i}| \le \frac{\kappa_{\mathrm{int}}}{n^4} + \varepsilon C_\beta \, \|\mathbf{J}_{:,i}\|_1.
    \]
\end{proof}

\begin{proof}[Proof of Theorem~\ref{cor:1}]
For uniformly spaced knots, each not-a-knot cubic spline $s_i$ satisfies:
\begin{itemize}
    \item $s_i$ is supported on at most 4 consecutive intervals $[t_{i-2},t_{i+2}]$,
    \item $\|s_i\|_\infty = \mathcal{O}(1)$ uniformly in $i$,
    \item $\int_{t_0}^{t_n} |s_i(t)| dt = \mathcal{O}(h)$,
    \item $\|s_i'\|_\infty = \mathcal{O}(h^{-1})$.
\end{itemize}

\textbf{Differentiation matrix $\mathbf{L}$.}
\emph{Row norms:} $\mathbf{L}_{i,:} = [s_i'(t_0),\dots,s_i'(t_n)]$. Each spline is nonzero at at most 4 knots, each derivative contributes $O(h^{-1})$, so
    \[
    \|\mathbf{L}_{i,:}\|_1 = \sum_{k=0}^n |s_i'(t_k)| = \mathcal{O}(4h^{-1}) = \mathcal{O}(h^{-1}) = \mathcal{O}(n).
    \]
\emph{Column norms:} $\mathbf{L}_{:,k} = [s_0'(t_k),\dots,s_n'(t_k)]^T$. Only ~4 splines are nonzero at a given knot, each one contributes $O(h^{-1})$ to the sum, so
    \[
        \|\mathbf{L}_{:,k}\|_1 = O(h^{-1}) = O(n).
    \]

\textbf{Integration matrix $\mathbf{J}$.}
\emph{Row norms:} $\mathbf{J}_{i,:} = \Big[\int_{t_0}^{t_0}s_i, \dots, \int_{t_0}^{t_n}s_i\Big]$. 
    The integral grows only over the support of $s_i$, which has total length $O(h)$, then remains constant. Summing over $n+1$ knots gives
    \[
    \|\mathbf{J}_{i,:}\|_1 = \sum_{k=0}^n \Big|\int_{t_0}^{t_k} s_i(t) dt\Big| = O(n h) = O(t_n-t_0).
    \]
\emph{Column norms:} $\mathbf{J}_{:,k} = \Big[\int_{t_0}^{t_k}s_0, \dots, \int_{t_0}^{t_k}s_n\Big]^T$. At a given $t_k$, roughly $k$ splines contribute, each $O(h)$, so
    \[
    \|\mathbf{J}_{:,k}\|_1 = O(k h) \le O(n h) = O(t_n-t_0).
    \]
All bounds are uniform in $i$, which completes the proof.
\end{proof}

\end{document}